\documentclass[a4paper, reqno, 10pt]{amsart}

\usepackage{verbatim}
\usepackage[dvips]{color}
\usepackage[usenames,dvipsnames]{xcolor}
\usepackage{amssymb}
\usepackage[active]{srcltx}
\usepackage[colorlinks,pdfpagelabels,pdfstartview = FitH,bookmarksopen = true,bookmarksnumbered = true,linkcolor = NavyBlue,plainpages = false,hypertexnames = false,citecolor = OliveGreen] {hyperref}
\usepackage[italian, textsize=tiny, color=BlueGreen, backgroundcolor=BlueGreen, linecolor=BlueGreen]{todonotes}
\usepackage{mathtools}
\usepackage{empheq}
\usepackage{a4wide}
\usepackage{aligned-overset}
\usepackage{kotex}

\newcommand\blfootnote[1]{%
\begingroup
\renewcommand\thefootnote{}\footnote{#1}%
\addtocounter{footnote}{-1}%
\endgroup
}

\makeatletter
\@namedef{subjclassname@2020}{\textup{2020} Mathematics Subject Classification}
\makeatother

\allowdisplaybreaks

\title[Nonlocal equations with kernels of general order]{Nonlocal equations with kernels of general order}

\author[Ok]{Jihoon Ok}
\address{Department of Mathematics, Sogang University, Seoul 04107, Republic of Korea}
\email{jihoonok@sogang.ac.kr}

\author[Song]{Kyeong Song}
\address{School of Mathematics, Korea Institute for Advanced Study, Seoul 02455, Republic of Korea}
\email{kyeongsong@kias.re.kr}

\subjclass[2020]{
35R11; % Fractional partial differential equations
35B65;  % Smoothness and regularity of solutions
47G20; % Integro-differential operators
35D30; 	% Weak solutions to PDEs
35R05. % PDEs with low regular coefficients and/or low regular data
}

\keywords{Nonlinear nonlocal equation; general order; regularity; Harnack inequality; fractional Sobolev space}

\newtheorem{theorem}{Theorem}[section]

\newtheorem{proposition}[theorem]{Proposition}
\newtheorem{lemma}[theorem]{Lemma}

\newtheorem{corollary}[theorem]{Corollary}
\theoremstyle{definition}

\newtheorem{remark}[theorem]{Remark}%[section]

\numberwithin{equation}{section}

\def\eqn#1$$#2$${\begin{equation}\label#1#2\end{equation}}
\def\charfn_#1{{\raise1.2pt\hbox{$\chi_{\kern-1pt\lower3pt\hbox{{$\scriptstyle#1$}}}$}}}

\makeatletter
\newcommand{\pushright}[1]{\ifmeasuring@#1\else\omit\hfill$\displaystyle#1$\fi\ignorespaces}
\newcommand{\pushleft}[1]{\ifmeasuring@#1\else\omit$\displaystyle#1$\hfill\fi\ignorespaces}
\makeatother

\DeclareMathOperator*{\osc}{osc}
\def\diam{\operatorname{diam}}

\newcommand{\tail}{{\rm Tail}}

\def\R{\mathbb R}

\def\loc{{\operatorname{loc}}}

\newcommand{\supp}{{\rm supp}}

\def\mean#1{\mathchoice%
          {\mathop{\kern 0.2em\vrule width 0.6em height 0.69678ex depth -0.58065ex
                  \kern -0.8em \intop}\nolimits_{\kern -0.4em#1}}%
          {\mathop{\kern 0.1em\vrule width 0.5em height 0.69678ex depth -0.60387ex
                  \kern -0.6em \intop}\nolimits_{#1}}%
          {\mathop{\kern 0.1em\vrule width 0.5em height 0.69678ex
              depth -0.60387ex
                  \kern -0.6em \intop}\nolimits_{#1}}%
          {\mathop{\kern 0.1em\vrule width 0.5em height 0.69678ex depth -0.60387ex
                  \kern -0.6em \intop}\nolimits_{#1}}}

\newcommand{\vertiii}[1]{{\left\vert\kern-0.25ex\left\vert\kern-0.25ex\left\vert #1 
			\right\vert\kern-0.25ex\right\vert\kern-0.25ex\right\vert}}
			
\def\avenorm#1{\mathchoice%
          {\mathop{\kern 0.2em\vrule width 0.6em height 0.69678ex depth -0.58065ex
                  \kern -0.545em \|{#1}\|}}%
          {\mathop{\kern 0.1em\vrule width 0.5em height 0.69678ex depth -0.60387ex
                  \kern -0.495em \|{#1}\|}}%
          {\mathop{\kern 0.1em\vrule width 0.5em height 0.69678ex depth -0.60387ex
                  \kern -0.495em \|{#1}\|}}%
          {\mathop{\kern 0.1em\vrule width 0.5em height 0.69678ex depth -0.60387ex
                  \kern -0.495em \|{#1}\|}}}

\newtoks\by
\newtoks\paper
\newtoks\book
\newtoks\jour
\newtoks\yr
\newtoks\pages
\newtoks\vol
\newtoks\publ

\def\ota{{\hbox{\bf ???}}}
\def\cLear{\by=\ota\paper=\ota\book=\ota\jour=\ota\yr=\ota
\pages=\ota\vol=\ota\publ=\ota}
\def\endpaper{\the\by, \textit{\the\paper},
{\the\jour} \textbf{\the\vol} (\the\yr), \the\pages.\cLear}
\def\endbook{\the\by, \textit{\the\book},
\the\publ, \the\yr.\cLear}
\def\endpap{\the\by, \textit{\the\paper}, \the\jour.\cLear}
\def\endproc{\the\by, \textit{\the\paper}, \the\book, \the\publ,
\the\yr, \the\pages.\cLear}

\begin{document}

\begin{abstract}
We consider a broad class of nonlinear integro-differential equations with a kernel whose differentiability order is described by a general function $\phi$. This class includes not only the fractional $p$-Laplace equations, but also borderline cases when the fractional order approaches~$1$. Under mild assumptions on $\phi$, we establish  sharp Sobolev-Poincar\'e type inequalities for the associated Sobolev spaces, which are connected to a question raised by Brezis (Russian Math. Surveys 57:693--708, 2002). Using these inequalities, we prove H\"older regularity and Harnack inequalities for weak solutions to such nonlocal equations. All the estimates in our results remain stable as the associated nonlocal energy functional approaches its local counterpart. 
\end{abstract}

\maketitle

\blfootnote{
J. Ok was supported by the National Research Foundation of Korea(NRF) grant
funded by the Korea government(MSIT) (NO. 2022R1C1C1004523).  K. Song was supported by a KIAS individual grant (MG091702) at Korea Institute for Advanced Study.}

\section{Introduction}
In this paper, we study nonlinear integro-differential equations of the form
\begin{equation}\label{mainPDE}
\mathcal{L}u(x) \coloneqq \mathrm{P.V.} \int_{\mathbb{R}^{n}} |u(x)-u(y)|^{p-2} (u(x)-u(y)) K(x,y)\, dy =0  \quad  \text{in }\ \Omega,
\end{equation}
where $\Omega\subset \mathbb{R}^n$ is a bounded domain, $p \in (1,\infty)$ is a growth exponent, and $K: \mathbb{R}^{n}\times \mathbb{R}^{n} \to \mathbb{R}$ is a measurable, symmetric kernel that satisfies
\begin{equation}\label{kernel}
\Lambda^{-1}
\frac{\phi(|x-y|)}{ |x-y|^{n+p}}  \le K(x,y) \le \Lambda  \frac{\phi(|x-y|)}{|x-y|^{n+p}}  \quad \text{for a.e. }\ x,y\in \mathbb{R}^{n},
\end{equation}
for a constant $\Lambda \ge 1$.  
Assumptions on the function $\phi:[0,\infty)\to[0,\infty)$ will be given in Section~\ref{result} below.

Notably, if $K(x,y)=\phi(|x-y|)|x-y|^{-n-p} = |x-y|^{-n-s p}$ (i.e., $\phi(t)=t^{(1-s)p}$) for a constant $s \in(0,1)$, then the equation \eqref{mainPDE} becomes the ($s$-)fractional $p$-Laplace equation $(-\Delta_{p})^{s}=0$, and moreover when $p=2$, it reduces to the fractional Laplace equation $(-\Delta)^{s}u=0$. 
There have been significant developments in the theory of nonlocal equations in the last two decades, especially since the work of Caffarelli and Silvestre \cite{CS07}. We refer to \cite{FRRO} and references therein for various results for linear nonlocal equations of the fractional Laplacian type. In particular, H\"older regularity  and  Harnack inequalities for linear nonlocal equations of the form \eqref{mainPDE} with $p=2$ and $K$ satisfying the uniform ellipticity condition $K(x,y)\approx |x-y|^{-n-2s}$ have been established in \cite{CCV11,Kas07,Kas09,Kas11}.

One of the major issues in the theory of nonlocal operators is to find a general class of kernels under which these regularity results still hold. 
From a probabilistic perspective, linear nonlocal operators with kernels satisfying more general differentiability conditions  have also been actively studied. 
Bass and Kassmann \cite{BK05a,BK05b} and Silvestre \cite{Sil} considered  measurable kernels satisfying  $|x-y|^{-n-2s_1}\lesssim K(x,y)\lesssim |x-y|^{-n-2s_2}$ for some $0<s_1<s_2<1$ or $K(x,y)\approx |x-y|^{-n-2s(x)}$, and proved H\"older regularity and Harnack inequalities for bounded solutions to the linear nonlocal equations \eqref{mainPDE} with  $p=2$ and such kernels $K(x,y)$. Moreover, regularity theory for the linear nonlocal equations with kernels  of the form $K(x,y)\approx \psi(|x-y|)^{-1}|x-y|^{-n}$, where $\psi$ satisfies the so-called weak scaling condition
\begin{equation}\label{psicondition}
L^{-1}\left(\frac{t_2}{t_1}\right)^{\alpha_1} \le \frac{\psi(t_2)}{\psi(t_1)} \le L \left(\frac{t_2}{t_1}\right)^{\alpha_2} \quad 
\text{for any }\ 0< t_1 \le t_2,
\end{equation}
for some $0<\alpha_1 \le \alpha_2<2$, has been studied in, for instance, \cite{Bae15,BK15,KK15,KKK13,KKLL}. This condition on $\psi$ naturally arises in the study of L\'evy processes. A major open problem in this area is concerned with establishing regularity estimates for nonlocal equations corresponding to non-L\'evy processes, particularly when the weak scaling indices $\alpha_1$ and $\alpha_2$ are not necessarily strictly less than $2$ (cf. \cite{BKKL19,KS15,Sz17}). In this regard, the authors of \cite{BKKL19} obtained heat kernel estimates when the kernel $K(x,y)\approx \psi(|x-y|)^{-1}|x-y|^{-n}$ satisfies \eqref{psicondition} for some $0< \alpha_1\le \alpha_2$ and $\int_0^1 (t/\psi(t)) \,dt <\infty$, and, using these, proved that any bounded parabolic function is H\"older continuous. 

In this paper, we consider a general class of kernels $K$ with a mild condition, which covers the one in \cite{BKKL19}. We establish the De Giorgi-Nash-Moser theory, including local boundedness, H\"older continuity and Harnack inequalities with sharp estimates involving a nonlocal tail term, and without assuming the a priori boundedness of the solution. Moreover, we present a purely analytic proof, independent of probability theory, that is naturally extended to the nonlinear equation \eqref{mainPDE} with $p$-growth, and could be applicable to corresponding parabolic nonlocal problems.

The De Giorgi-Nash-Moser theory for weak solutions to the nonlinear nonlocal equations \eqref{mainPDE}  with measurable kernel $K(x,y)\approx |x-y|^{-n-s p}$ was first established by Di Castro, Kuusi, and Palatucci \cite{DKP14,DKP16}. 
Cozzi \cite{Coz17} further extended these results to minimizers of a larger class of non-differentiable functionals, by introducing the notion of the fractional De Giorgi class. Furthermore, the regularity estimates obtained in \cite{Coz17} remain stable as the differentiability parameter $s$ approaches to $1$, which is consistent with the convergence of the associated $W^{s,p}$-energies shown in \cite{BBM01,Bre02}; see also \cite{IN10} for the limiting behavior of viscosity solutions. 
For various issues involving the $s$-fractional $p$-Laplacian, see \cite{BDOR,BK,DN,KKL19,KMS15a,Lin16,Sch} and references therein. We further refer to \cite{BKO,BKS,BOS,CK,CKW22,CKW23,DP19,Ok23} for regularity results for nonlocal equations with non-standard growth.

Here we outline the main points of this paper, along with a brief discussion on our assumptions and techniques. 
Equation \eqref{mainPDE} is modeled on the example
\[ \mathrm{P.V.} \int_{\mathbb{R}^{n}} |u(x)-u(y)|^{p-2} (u(x)-u(y)) \frac{\phi(|x-y|)}{|x-y|^{n+p}}\, dy =0  \quad  \text{in }\ \Omega, \]
that is, the case $K(x,y)=\phi(|x-y|)|x-y|^{-n-p}$. 
A direct calculation shows that this is the Euler-Lagrange equation of the energy functional
\begin{equation}\label{model.ftnal}
w \mapsto \int_{\mathbb{R}^{n}}\int_{\mathbb{R}^{n}} \frac{|w(x)-w(y)|^p}{|x-y|^p}\frac{\phi(|x-y|)}{|x-y|^n} \,dy\,dx.
\end{equation}
The fractional energies of the form \eqref{model.ftnal},
 which generalize the usual Gagliardo seminorms of $W^{s,p}$-functions, were studied by Bourgain, Brezis, and Mironescu \cite{BBM01} in the characterization of the classical Sobolev space $W^{1,p}$; see also \cite{Bre02} for a  different approach. Related Poincar\'e type inequalities and their applications to the Ginzburg-Landau model can be found in \cite{DGM,Ponce} and \cite{BBM04}, respectively. Note that we will assume the Dini type condition \eqref{Dini} on the function $\phi$. If this condition fails, then the energy \eqref{model.ftnal} is finite only for constant functions, see \cite[Proposition~3]{Bre02}. Therefore,  this condition essentially needed. Function spaces linked to \eqref{model.ftnal} will be discussed in Section~\ref{space} below. 
Taking into account the aforementioned literature on nonlocal equations and function spaces, a natural question is whether similar regularity results can be obtained for more general equations of the type \eqref{mainPDE} with the kernel $K$ and the function $\phi$ satisfying \eqref{kernel} and \eqref{Dini}, respectively.  However, we could not find any results in this direction.

In order to obtain our regularity results, we assume that $\phi$ satisfies \eqref{Dini} and \eqref{adec} below. As mentioned above, the first condition is essential. The second condition  \eqref{adec} is a coercivity condition and similar conditions have been assumed even in the case of linear nonlocal equations \cite{KKK13,KKLL}. Therefore, our regularity results can unify most results for nonlocal equations of the form \eqref{mainPDE} with \eqref{kernel}, and even can cover borderline cases as $s$ approaches $1$ such as
\[
K(x,y) = \max\left\{\frac{(-\log |x-y|)^{-\gamma}}{|x-y|^{n+p}}, \frac{1}{|x-y|^{n+sp}}\right\}, \quad \gamma>1  \ \text{ and }\ s\in (0,1).
\]  
To our best knowledge, each of our results is the first one for nonlinear nonlocal equations with $p$-growth and general differentiability order including the above borderline case. 
We also emphasize that our regularity estimates remain stable as the nonlocal functional \eqref{model.ftnal} approaches the local $p$-Dirichet functional, see Remark~\ref{rmk.conv} and Proposition~\ref{prop.energy} below. 
To this aim, the key step is to establish a sharp Sobolev-Poincar\'e type inequality for generalized fractional Sobolev spaces related to \eqref{model.ftnal} (see Theorem~\ref{thm:SoboPoin} and Corollary~\ref{cor.SP} below), which is also stable as the lower bound of the differentiability order approaches $1$.
This is also a new result of its own interest, which gives a potential resolution to \cite[Problem~4]{Bre02}. In order to obtain this inequality, we improve the method used in \cite{BDOR}, where the authors obtained a Sobolev-Poincar\'e inequality for weighted fractional Sobolev spaces that is also stable as the differentiability order approaches $1$. We first derive fractional Riesz-type potential estimates involving the function $\phi$ with sharp constant with respect to $\phi$, and then apply a well-known Hardy-Littlewood-Sobolev inequality. Once we obtain a sharp Sobolev-Poincar\'e estimate, 
we can extend the approaches of \cite{DKP14,DKP16} to the setting of our problem, proving various fundamental estimates such as Caccioppoli estimates, logarithmic estimates and expansion of positivity. 

Let us introduce our main regularity results.

\subsection{Assumptions and main results}\label{result}
Let $\phi : [0,\infty)\to [0,\infty)$ be a measurable function satisfying $\phi(0)=0$ and $\phi(t)>0$ for $t>0$.  
We  assume that the following Dini condition holds:
\begin{equation}\label{Dini}
\int_{0}^{1}\phi(t)\,\frac{dt}{t} < \infty.
 \end{equation}
We also assume that there exist constants $0<s<\tilde{s}<\infty$ such that
the map $t\mapsto \phi(t)/t^{(1-s)p}$ is almost decreasing on $(0,\infty)$ and the map $t\mapsto \phi(t)/t^{(1-\tilde{s})p}$ is almost increasing on $(0,\infty)$,
 i.e., there exists a constant $L\ge1$ such that
\begin{equation}\label{adec}
\frac{\phi(t_{2})}{t_{2}^{(1-s)p}} \le L \frac{\phi(t_{1})}{t_{1}^{(1-s)p}}
\end{equation}
and
\begin{equation}\label{ainc}
\frac{\phi(t_{2})}{t_{2}^{(1-\tilde{s})p}} \ge L^{-1}\frac{\phi(t_{1})}{t_{1}^{(1-\tilde{s})p}}
\end{equation}
for any $0<t_{1} \le t_{2}<\infty$. 
These conditions are equivalent to 
\begin{equation}\label{adecequiv}
L^{-1}\lambda^{(1-\tilde{s})p}\phi(t) \le \phi(\lambda t)  \le L \lambda^{(1-s)p} \phi(t) \quad \text{for all }\ t> 0 \ \text{ and } \ \lambda\ge 1.
\end{equation}

\begin{remark}\label{rmk:adec}
The almost decreasing condition \eqref{adec} implies the same condition with any $q> (1-s)p$ in place of $(1-s)p$. 
In the same way, the almost increasing condition \eqref{ainc} implies the same condition with any $\tilde{q} < (1-\tilde{s})p$ in place of $(1-\tilde{s})p$. 
Hence, throughout this paper, we may assume that %$s\in(0,1)$ is sufficiently small such that $s p < n$ and that $\tilde{s}>1$. 
\[ 0 < s < \min\{1,n/p\} \quad \text{and} \quad \tilde{s}>1. \]
\end{remark}

In light of \eqref{Dini}, we are further able to define $\Phi:[0,\infty)\to [0,\infty)$ by $\Phi(0)\coloneqq0$ and
\begin{equation}\label{def.Phi}
\Phi(t) \coloneqq \int_{0}^{t}\phi(\tau)\,\frac{d\tau}{\tau}  \qquad \text{for any}\;\; t \in (0,\infty).
\end{equation} 
Then $\Phi$ is nondecreasing and continuous on $[0,\infty)$ and differentiable on $(0,\infty)$. It also satisfies \eqref{adec} and \eqref{ainc}, hence  \eqref{adecequiv}, with $\phi$ replaced by $\Phi$. The inequality \eqref{adecequiv} for $\phi$ and also $\Phi$  will be frequently used throughout this paper. Moreover, we obtain from \eqref{adec} that
\begin{equation}\label{phi.Phi} 
\Phi(t) = \int_0^t \phi(\tau)\frac{d\tau}{\tau} \ge    \frac{\phi(t)}{L t^{(1-s)p}}  \int_0^t \tau^{(1-s)p}\frac{d\tau}{\tau} = \frac{\phi(t)}{L(1-s)p} > \frac{\phi(t)}{Lp} 
 \qquad \text{for all }\ t > 0. 
\end{equation}

The equation  \eqref{mainPDE} with $K$ and $\phi$ satisfying \eqref{kernel}, \eqref{Dini}, \eqref{adec} and \eqref{ainc} covers the following examples: 
\begin{itemize}
\item[(i)] For constants $0<s< \tilde{s}<1$ and $L\ge 1$, $\phi$ satisfies that 
\[
L^{-1} \left(\frac{t_2}{t_1}\right)^{(1-\tilde{s})p} \le \frac{\phi(t_2)}{\phi(t_1)} \le L \left(\frac{t_2}{t_1}\right)^{(1-s)p} \quad 
\text{for any }\ 0<t_1 \le t_2 <\infty.
\]
This model with $p=2$ has been studied in, for instance, \cite{Bae15,BK15,BK05a,BK05b,KKLL,Sil}. The following are specific examples satisfying the previous inequality: 
\[  
 t^{(1-s)p}, \quad  t^{(1-s)p} + t^{(1-\tilde{s})p}, \quad \min\big\{t^{(1-s)p},t^{(1-\tilde{s})p}\big\}, \quad t^{(1-s)p} [\log\left(1+t^{-1}\right)]^\gamma \ \text{ for } \ \gamma>0.\]
 
 \item[(ii)] For constants $0<s<1<\gamma$, 
\[
\phi(t)=\max \left\{ (- \log t)^{-\gamma}, t^{(1-s)p} \right\}.
\]
In particular, the second assumption \eqref{adec} holds with $L=1$.
Observe that this model is not covered by the one in (i). The case $p=2$ has been studied in \cite{BKKL19}.

\item[(iii)] Let $\psi(t) \coloneqq t^{p}/\phi(t)$. Then $K(x,y) \approx \psi(|x-y|)^{-1}|x-y|^{-n}$, and \eqref{Dini}--\eqref{ainc} imply that
\[
\int_0^1\frac{t^{p-1}}{\psi(t)}\, dt <\infty,
\quad\text{and}\quad
L^{-1}\left(\frac{t_2}{t_1}\right)^{sp} \le \frac{\psi(t_2)}{\psi(t_1)} \le L\left(\frac{t_2}{t_1}\right)^{\tilde{s}p} \ \ \text{for }\ 0<t_1 \le t_2 < \infty.
\]
In particular, when $p=2$, the latter inequality reduces to \eqref{psicondition}. 
\end{itemize}

With the function space $\mathbb{W}^{\phi,p}(\Omega)$ to be introduced in the next section, here we define weak solutions. We say that a function $u \in \mathbb{W}^{\phi,p}(\Omega)$ is a weak supersolution (resp. subsolution) to the integro-differential equation 
\eqref{mainPDE}
 if 
 \begin{equation*}
\iint_{\mathcal{C}_{\Omega}}|u(x)-u(y)|^{p-2}(u(x)-u(y))(\zeta(x)-\zeta(y))K(x,y)\,dx\,dy \ge 0 \;\; (\text{resp. } \le 0) 
\end{equation*}
for every $\zeta \in \mathbb{W}^{\phi,p}(\Omega)$ with $\zeta \ge 0$ a.e. in $\mathbb{R}^{n}$ and $\zeta = 0$ a.e. in $\mathbb{R}^{n}\setminus\Omega$, where 
\begin{equation}\label{mcdef} 
\mathcal{C}_{\Omega} \coloneqq (\mathbb{R}^{n}\times \mathbb{R}^{n}) \setminus ((\mathbb{R}^{n}\setminus \Omega)\times (\mathbb{R}^{n}\setminus \Omega)).
\end{equation}
Moreover, we say that $u \in \mathbb{W}^{\phi,p}(\Omega)$ is a weak solution to \eqref{mainPDE} if it is both a weak supersolution and a weak subsolution to \eqref{mainPDE}. 
Existence and uniqueness of weak solutions to Dirichlet problems involving \eqref{mainPDE} will be discussed in Section~\ref{subsec3} below. 

Given a function $f:\mathbb{R}^{n}\rightarrow\mathbb{R}$, we consider its nonlocal tail defined by
\begin{equation}\label{def.tail}
\tail(f;x_{0},r) \coloneqq \left(\frac{r^{p}}{\Phi(r)}\int_{\mathbb{R}^{n}\setminus B_{r}(x_{0})}|f(y)|^{p-1}\frac{\phi(|y-x_{0}|)}{|y-x_{0}|^{n+p}}\,dy\right)^{1/(p-1)} 
\end{equation}
whenever $x_{0}\in\mathbb{R}^{n}$ and $r>0$, where $\Phi$ is given in \eqref{def.Phi}. If there is no confusion, we will omit the point $x_{0}$ and simply write $\tail(f;x_{0},r) \equiv \tail(f;r)$. 
It can be shown that every function in $\mathbb{W}^{\phi,p}(\Omega)$ has finite tails, see Lemma~\ref{lem.tail.fin} below. 

We now state our main results.

\begin{theorem}[Local boundedness]\label{thm.bdd}
Let $u \in \mathbb{W}^{\phi,p}(\Omega)$ be a weak subsolution to \eqref{mainPDE} under assumptions \eqref{kernel}, \eqref{Dini}, \eqref{adec} and \eqref{ainc}. Then we have
\begin{equation}\label{sup.est}
\sup_{B_{r/2}(x_{0})}u \le c_{b}\varepsilon^{-\frac{n(p-1)}{s p^2}}\left(\mean{B_{r}(x_{0})}u_{+}^{p}\,dx\right)^{1/p} + \varepsilon\tail(u_{+};x_{0},r/2)
\end{equation}
whenever $B_{r}(x_{0}) \subset \Omega$ is a ball and $\varepsilon \in (0,1]$, where $c_{b}= c_{b}(n,p,s,\tilde{s},\Lambda,L) \ge 1$ is a constant. Consequently, if $u$ is a weak solution to \eqref{mainPDE}, then $u \in L^{\infty}_{\loc}(\Omega)$ with the estimate
\begin{equation}\label{sup.est2}
\|u\|_{L^{\infty}(B_{r/2}(x_{0}))} \le c_{b}\varepsilon^{-\frac{n(p-1)}{s p^2}}\left(\mean{B_{r}(x_{0})}|u|^{p}\,dx\right)^{1/p} + \varepsilon\tail(u;x_{0},r/2).
\end{equation}
\end{theorem}

\begin{theorem}[H\"older regularity]\label{thm.hol}
Let $u \in \mathbb{W}^{\phi,p}(\Omega)$ be a weak solution to \eqref{mainPDE} under assumptions \eqref{kernel}, \eqref{Dini}, \eqref{adec} and \eqref{ainc}. Then $u \in C^{0,\alpha}_{\loc}(\Omega)$ for some $\alpha = \alpha(n,p,s,\tilde{s},\Lambda,L) \in (0,1)$. Moreover, we have
\begin{equation}\label{holder.est}
\osc_{B_{\rho}(x_{0})}u \le c\left(\frac{\rho}{r}\right)^{\alpha}\left[\left(\mean{B_{r}(x_{0})}|u|^{p}\,dx\right)^{1/p} + \tail(u;x_{0},r/2)\right]
\end{equation}
whenever $B_{r}(x_{0})\subset \Omega$ is a ball and $\rho \in (0,r/2]$, where $c= c(n,p,s,\tilde{s},\Lambda,L) \ge 1$ is a constant. 
\end{theorem}

\begin{theorem}[Nonlocal Harnack inequality]\label{thm.harnack}
Let $u \in \mathbb{W}^{\phi,p}(\Omega)$ be a weak solution to \eqref{mainPDE} under assumptions \eqref{kernel}, \eqref{Dini}, \eqref{adec} and \eqref{ainc}, which is nonnegative in a ball $B_{R}(x_{0}) \subset \Omega$. Then for any $r \in (0,R/2]$, we have
\begin{equation}\label{harnack.est}
\sup_{B_{r}(x_{0})}u \le c\inf_{B_{r}(x_{0})}u + c\left(\frac{r^{p}}{\Phi(r)}\frac{\Phi(R)}{R^{p}}\right)^{1/(p-1)}\tail(u_{-};x_{0},R),
\end{equation}
where $c = c(n,p,s,\tilde{s},\Lambda,L) \ge 1$ is a constant.
\end{theorem}

\begin{theorem}[Nonlocal weak Harnack inequality]\label{thm.wh}
Let $u \in \mathbb{W}^{\phi,p}(\Omega)$ be a weak supersolution to \eqref{mainPDE} under assumptions \eqref{kernel},  \eqref{Dini}, \eqref{adec} and \eqref{ainc}, which is nonnegative in a ball $B_{R}(x_{0}) \subset \Omega$. Let 
\begin{equation*}
\bar{t} \coloneqq 
\begin{cases}
\frac{n(p-1)}{n-s p} & \text{if }\ s p <n, \\
\infty & \text{if }\ s p \ge n.
\end{cases}
\end{equation*}
Then for any $r \in (0,R/2]$ and $t \in (0,\bar{t})$, we have
\begin{equation}\label{weak.harnack.est}
\left(\mean{B_{r/2}(x_{0})}u^{t}\,dx\right)^{1/t} \le c\inf_{B_{r}(x_{0})}u + c\left(\frac{r^{p}}{\Phi(r)}\frac{\Phi(R)}{R^{p}}\right)^{1/(p-1)}\tail(u_{-};x_{0},R),
\end{equation}
where $c= c(n,p,s,\tilde{s},\Lambda,L,t) \ge 1$ is a constant.
\end{theorem}

\begin{remark}\label{rmk.conv}
We highlight that, by carefully analyzing the factors related to $\phi$ in several estimates, we are able to show the robustness of our results. 
More precisely, the estimates given in Theorems~\ref{thm.bdd}--\ref{thm.wh} are stable as $s\nearrow 1$ and recover the estimates available for the local $p$-Laplacian in the limit case. 
Indeed, if $0<s_{0}<s<1$, then Remark~\ref{rmk:adec} implies that $\phi$ satisfies \eqref{adec} with $s$ replaced by $s_{0}$, which allows us to obtain each of \eqref{sup.est}--\eqref{weak.harnack.est} with the constant $c$ depending on $s_{0}$ instead of $s$. 
Also note that if $f \in L^{q}(\mathbb{R}^{n})$ for some $q \ge p-1$ and $r \in (0,\infty)$, then \eqref{adec} and \eqref{phi.Phi} imply
\begin{equation*}
\begin{aligned}
[\tail(f;x_{0},r)]^{p-1} & = \frac{r^{p}}{\Phi(r)}\int_{\mathbb{R}^{n}\setminus B_{r}(x_{0})}|f(y)|^{p-1}\frac{\phi(|y-x_{0}|)}{|y-x_{0}|^{n+p}}\,dy \\
& \le L\frac{\phi(r)}{\Phi(r)}r^{s p}\int_{\mathbb{R}^{n}\setminus B_{r}(x_{0})}\frac{|f(y)|^{p-1}}{|y-x_{0}|^{n+s p}}\,dy \\
& \le L^{2}(1-s)p\left(1+\frac{1+|x_{0}|}{r}\right)^{s p}r^{s p}\int_{\mathbb{R}^{n}}\frac{|f(y)|^{p-1}}{(1+|y|)^{n+s p}}\,dy \; \longrightarrow \; 0 \quad \text{as }\ s\nearrow1.
\end{aligned}
\end{equation*}
\end{remark}

\begin{remark}
We give a few comments on possible extensions of our results.
\begin{itemize}
\item[(i)]
The symmetry condition on the kernel $K$ is not really restrictive, since otherwise we may consider $\tilde{K}(x,y) \coloneqq \frac{1}{2}(K(x,y)+K(y,x))$. 

\item[(ii)]
Moreover, similar results on local regularity and Harnack inequality can be obtained under the following weaker assumption on $K$: there exists $r_{0}>0$ such that
\begin{equation*} 
 \Lambda^{-1}\frac{\phi(|x-y|)}{ |x-y|^{n+p}} \le K(x,y) \le \Lambda  \frac{\phi(|x-y|)}{|x-y|^{n+p}} \quad \text{for a.e. } x,y\in\mathbb{R}^{n} \text{ with } |x-y|<r_{0}
\end{equation*}
and 
\begin{equation*}
\sup_{x\in\mathbb{R}^{n}}\int_{\{|y-x|\ge r_{0}\}}K(x,y)\,dy + \sup_{\substack{x,y\in\mathbb{R}^{n} \\ |x-y| \ge r_{0}}}K(x,y) < \infty,
\end{equation*}
see for instance \cite{Bae15,Kas09,Kas11}. However, in order to obtain precise tail terms and make our estimates stable as $s \nearrow 1$, we will assume \eqref{kernel} in this paper.

\item[(iii)]
It is possible to extend Theorems~\ref{thm.bdd}--\ref{thm.wh} to inhomogeneous equations of the form
\[  \mathrm{P.V.} \int_{\mathbb{R}^{n}} |u(x)-u(y)|^{p-2}(u(x)-u(y))K(x,y)\, dy = f(x,u) \quad  \text{in } \Omega, \]
under a suitable growth condition on $f$. 
In this case, additional terms responsible for $f$ will appear in estimates \eqref{sup.est}--\eqref{weak.harnack.est}, but the proofs are almost the same as in the case of \eqref{mainPDE}; see for instance \cite{Coz17}.
\end{itemize}
\end{remark}

The paper is organized as follows. In Section~\ref{sec2}, we introduce notations and basic properties of function spaces, and then prove the existence of weak solutions to \eqref{mainPDE}. In Section~\ref{sec3}, we prove Sobolev-Poincar\'e type inequalities for the function space $W^{\phi,p}$. In Section~\ref{sec4}, we derive Caccioppoli and logarithmic estimates and then prove Theorem~\ref{thm.bdd}. Finally, in Section~\ref{sec5}, we show an expansion of positivity lemma to prove Theorems~\ref{thm.hol}--\ref{thm.wh}.

\section{Preliminaries}\label{sec2}
\subsection{Notation}
Throughout this paper, we denote by $c$ a general positive constant, whose value may vary from line to line. Specific dependencies of constants are denoted by using parentheses.
As usual, 
\[ B_{r}(x_{0}) \coloneqq \left\{ x\in\mathbb{R}^{n} : |x-x_{0}| < r \right\} \]
denotes the $n$-dimensional open ball with center $x_{0} \in \mathbb{R}^{n}$ and radius $r>0$. 
If there is no confusion, we omit the center and simply write $B_{r} \equiv B_{r}(x_{0})$. 
Moreover, given a ball $B$, we denote by $\gamma B$ the concentric ball with radius magnified by a factor $\gamma>0$. 
Unless otherwise stated, different balls in the same context are concentric. 
We denote by $\omega_{n} \coloneqq 2\pi^{n/2}/\Gamma(n/2)$ the surface area of the $(n-1)$-dimensional unit sphere $\partial B_1$. 

For a measurable function $f$, we write $f_{\pm}\coloneqq \max\{\pm f,0\}$. If $f$ is integrable over a measurable set $U\subset \mathbb{R}^{n}$ with $0<|U|<\infty$, we denote its integral average over $U$ by
\begin{equation*}
(f)_{U} \coloneqq \mean{U}f\,dx  \coloneqq \frac{1}{|U|} \int_{U}f\,dx .
\end{equation*}

\subsection{A variant of fractional Sobolev spaces}\label{space}
Let $U\subseteq \mathbb{R}^{n}$ be an open set. For the function $\phi$ described in the previous section and a constant  $p\in (1,\infty)$,  we consider a fractional Sobolev space of general order defined as
\[ W^{\phi,p}(U) \coloneqq \left\{f\in L^{p}(U): [f]_{W^{\phi,p}(U)}^p \coloneqq \int_{U}\int_{U}\frac{|f(x)-f(y)|^{p}}{|x-y|^{p}}\frac{\phi(|x-y|)}{|x-y|^{n}}\,dy\,dx<\infty\right\}. \] 
Observe that $W^{\phi,p}(U)$ is a Banach space equipped with the norm
\[ \|f\|_{W^{\phi,p}(U)} \coloneqq \|f\|_{L^{p}(U)} + [f]_{W^{\phi,p}(U)}. \]
We also note that, in the particular case when $\phi(t)=t^{(1-s)p}$ for a constant $s \in (0,1)$, the space $W^{\phi,p}(U)$ reduces to the Sobolev-Slobodeckij space $W^{s,p}(U)$.

We introduce another function space related to weak solutions to \eqref{mainPDE}. We define $\mathbb W^{\phi,p}(\Omega)$ as the set of all measurable functions $f:\R^n \to \R$ satisfying 
\[
f|_{\Omega}\in L^{p}(\Omega) 
\quad \text{and}\quad
\iint_{\mathcal C_\Omega} \frac{|f(x)-f(y)|^{p}}{|x-y|^{p}}\frac{\phi(|x-y|)}{|x-y|^{n}} \,dy\, dx  < \infty,
\]
where $\mathcal{C}_{\Omega}$ is defined in \eqref{mcdef}. 
It is clear that $W^{\phi,p}(\mathbb{R}^{n}) \subset \mathbb{W}^{\phi,p}(\Omega)$ and that %$f|_{\Omega} \in W^{\phi,p}(\Omega)$ for any $f\in \mathbb W^{\phi,p}(\Omega)$.
\[ f \in \mathbb{W}^{\phi,p}(\Omega) \;\; \Longrightarrow \;\; f|_{\Omega} \in W^{\phi,p}(\Omega). \]

We note that the conditions \eqref{adec} and \eqref{ainc} directly imply the following estimates of the integral of the kernel $\phi(|x-y|)/|x-y|^{n+p}$ over $\R^n\setminus B_r(x)$.  
\begin{lemma}
Assume that $\phi : [0,\infty)\to [0,\infty)$ satisfies \eqref{adec}, and let $B_{r}(x) \subset \mathbb{R}^n$ be a ball. 
Then
\begin{equation}\label{out.int} 
\int_{\mathbb{R}^{n}\setminus B_{r}(x)}\frac{\phi(|x-y|)}{|x-y|^{n+p}}\,dy \le   L\frac{\omega_{n}}{s p}\frac{\phi(r)}{r^{p}}. 
\end{equation}
%\item[(ii)] If $\phi$ satisfies \eqref{ainc}, then 
%\begin{equation}\label{out.int2}
% \int_{\mathbb{R}^{n}\setminus B_{r}(x)}\frac{\phi(|x-y|)}{|x-y|^{n+p}}\,dy \ge L^{-1}\frac{\omega_{n}}{\tilde{s}p}\frac{\phi(r)}{r^{p}} .
%\end{equation}
%\end{itemize}
\end{lemma}

We prove embedding results for $W^{1,p}(\Omega)$, $W^{\phi,p}(\Omega)$ or $\mathbb{W}^{\phi,p}(\Omega)$, $W^{s,p}(\Omega)$. 
\begin{proposition}\label{prop:embedding}
Let $p\in(1,\infty)$ and $\Omega$ be a bounded open set in $\R^n$ with $R \coloneqq \mathrm{diam}(\Omega)$. Assume that $\phi:[0,\infty)\rightarrow[0,\infty)$ is a measurable function. 
\begin{itemize}
\item[(i)] If $\phi$ satisfies \eqref{Dini} and  $\Omega$ is a $W^{1,p}$-extension domain, then $W^{1,p}(\Omega)$ is continuously embedded into  $W^{\phi,p}(\Omega)$. Moreover, for any ball $B_{r}\subset\mathbb{R}^{n}$ and $f \in W^{1,p}(B_r)$, we have
\[
\int_{B_r}\int_{B_r}\frac{|f(x)-f(y)|^{p}}{|x-y|^{p}}\frac{\phi(|x-y|)}{|x-y|^{n}}\,dx\,dy \le c (n,p) \Phi(2r) \int_{B_r} |Df|^p\, dx .
\]
\item[(ii)] If $\phi$ satisfies \eqref{Dini} and \eqref{adec}, then $W^{1,p}_0(\Omega)$ is continuously embedded into  $\mathbb{W}^{\phi,p}(\Omega)$. Moreover, for any $f\in W^{1,p}_{0}(\Omega)$, we have
\[
\iint_{\mathcal{C}_{\Omega}}\frac{|f(x)-f(y)|^{p}}{|x-y|^{p}}\frac{\phi(|x-y|)}{|x-y|^{n}}\,dx\,dy \le c (n,p,s) \left(\Phi(R) \int_\Omega |Df|^p\, dx + \frac{\phi(R)}{R^p} \int_\Omega |f|^p\,dx \right).
\]
\item[(iii)] If $\phi$ satisfies \eqref{adec} for some $s \in (0,1)$, then $W^{\phi,p}(\Omega)$ is continuously embedded into  $W^{s,p}(\Omega)$. Moreover, for any $f\in W^{\phi,p}(\Omega)$, we have
\[
\int_{\Omega}\int_{\Omega}\frac{|f(x)-f(y)|^{p}}{|x-y|^{s p+n}}\,dx\,dy \le L \frac{R^{(1-s)p}}{\phi(R)^{(1-s)p}}\int_{\Omega}\int_{\Omega}\frac{|f(x)-f(y)|^{p}}{|x-y|^{p}}\frac{\phi(|x-y|)}{|x-y|^{n}}\,dx\,dy  .
\]
\end{itemize} 
\end{proposition}
\begin{proof}
(i) The embedding result is essentially well-known, see for instance \cite[Theorem~1]{BBM01}. Suppose that $f\in W^{1,p}(B_r) \cap C^{1}(\overline{B_{r}})$. Then, by Fubini's theorem and  the change of variables with $z=tx+(1-t)y$ and $\xi = x-y$ for each $t\in [0,1]$, we have
\[\begin{aligned}
\int_{B_r}\int_{B_r}\frac{|f(x)-f(y)|^{p}}{|x-y|^{p}}\frac{\phi(|x-y|)}{|x-y|^{n}}\,dx\,dy 
&\le \int_0^1 \int_{B_r}\int_{B_r}|Df(tx+(1-t)y)|^{p} \frac{\phi(|x-y|)}{|x-y|^{n}}\,dx\,dy\,dt \\
&\le \int_0^1 \int_{B_{2r}}\int_{B_r}  |Df(z)|^{p} \frac{\phi(|\xi|)}{|\xi|^{n}} t^n(1-t)^n \,dz\,d\xi \,dt\\ 
& \le c(n,p)  \Phi(2r) \int_{B_r} |Df(z)|^p\,dz.
\end{aligned}\]

(ii) Let $B_{r} \subset \mathbb{R}^{n}$ be a ball satisfying $\Omega\subset B_r$, and suppose that $f\in W^{1,p}_0(\Omega)\subset W^{1,p}_0(B_r)$. By \eqref{out.int}, we have
\[\begin{aligned}
\int_{\R^n\setminus B_{2r}}\int_{B_{2r}}\frac{|f(x)-f(y)|^{p}}{|x-y|^{p}}\frac{\phi(|x-y|)}{|x-y|^{n}}\,dx\,dy 
&\le \int_{B_r}|f(x)|^p\left[\int_{\R^n\setminus B_{r}(x)}\frac{\phi(|x-y|)}{|x-y|^{n+p}}\,dy\right]\,dx   \\
& \le c(n,p,s)  \frac{\phi(r)}{r^p} \int_{B_r} |f(x)|^p\,dx.
\end{aligned}\]
Then, the desired estimate in (ii) follows from the one in (i) and the previous inequality.

 The estimate in (iii) directly follows from \eqref{adec}. This completes the proof. 
\end{proof}

The following proposition is concerned with the limit of our energy as $s \nearrow 1$.

\begin{proposition}\label{prop.energy}
For $0<s<1<\tilde{s}$, let $\phi_{s}:[0,\infty)\to[0,\infty)$ be any measurable function satisfying \eqref{Dini}, \eqref{adec} and \eqref{ainc}, and define 
\[ \Phi_{s}(t) \coloneqq \int_{0}^{t}\phi_{s}(\tau)\,\frac{d\tau}{\tau}, \qquad t \in (0,\infty). \]
If $f\in L^{p}(\mathbb{R}^{n})$, then
\[ \frac{1}{\Phi_{s}(r)}\int_{\mathbb{R}^{n}}\int_{\mathbb{R}^{n}}\frac{|f(x)-f(y)|^{p}}{|x-y|^{p}}\frac{\phi_{s}(|x-y|)}{|x-y|^{n}} \,dy\,dx\; \longrightarrow \; c(n,p)\int_{\mathbb{R}^{n}}|Df|^{p}\,dx \quad \text{as }\ s\nearrow1 \]
for any $r >0$, with the understanding that the limit is equal to $\infty$ if $f\notin W^{1,p}(\mathbb{R}^{n})$.	
\end{proposition}	

\begin{proof}
In the setting of \cite[Theorem~2]{Bre02}, we let 
\[ \rho_{\varepsilon}(t) = \frac{1}{\Phi_{s}(r)}\frac{\phi_{s}(t)}{t^{n}}\chi_{(0,r)}(t) \quad \text{for }\ \varepsilon = 1-s. \] 
Then the theorem implies
	\[ \frac{1}{\Phi_{s}(r)}\iint_{\{|x-y|<r\}}\frac{|f(x)-f(y)|^{p}}{|x-y|^{p}}\frac{\phi_{s}(|x-y|)}{|x-y|^{n}}\,dy\,dx \; \longrightarrow \; c(n,p)\int_{\mathbb{R}^{n}}|Df|^{p}\,dx \quad \text{as }\ s\nearrow1. \]
On the other hand, by \eqref{out.int} and \eqref{phi.Phi}, we have
\begin{equation*} 
	\begin{aligned}
		& \frac{1}{\Phi_{s}(r)}\iint_{\{|x-y| \ge r\}}\frac{|f(x)-f(y)|^{p}}{|x-y|^{p}}\frac{\phi_{s}(|x-y|)}{|x-y|^{n}}\,dy\,dx \\
		& \le \frac{2^{p-1}}{\Phi_{s}(r)}\iint_{\{|x-y| \ge r\}}\frac{|f(x)|^{p}+|f(y)|^{p}}{|x-y|^{p}}\frac{\phi_{s}(|x-y|)}{|x-y|^{n}}\,dy\,dx \\
		& \le \frac{2^{p}}{\Phi_{s}(r)}\int_{\mathbb{R}^{n}}|f(x)|^{p}\left(\int_{\mathbb{R}^{n}\setminus B_{r}(x)}\frac{\phi_{s}(|x-y|)}{|x-y|^{n+p}}\,dy\right)\,dx \\
		& \le \frac{c(n,p,L)}{s r^{p}}\frac{\phi_{s}(r)}{\Phi_{s}(r)}\int_{\mathbb{R}^{n}}|f(x)|^{p}\,dx \\
		& \le \frac{1-s}{s}\frac{c(n,p,L)}{r^{p}}\int_{\mathbb{R}^{n}}|f(x)|^{p}\,dx
		\; \longrightarrow \; 0 \quad \text{as }\ s\nearrow1. 
	\end{aligned} 
\end{equation*} 
Combining the above two displays, we get the desired conclusion.
\end{proof}

We next recall the definition of nonlocal tail given in \eqref{def.tail}. The following lemma is an analog of \cite[Proposition~3.2]{CKW22}, see also \cite[Proposition~13]{DK19}. 
\begin{lemma}\label{lem.tail.fin}
Let $p\in(1,\infty)$, and assume that $\phi:[0,\infty)\rightarrow[0,\infty)$ satisfies \eqref{adec} and \eqref{ainc}. If $ f \in \mathbb{W}^{\phi,p}(\Omega)$, then $\tail(f;x_{0},r)<\infty$ for any ball $B_{r}(x_{0})\subset \Omega$. 
\end{lemma}
\begin{proof}
Let us assume $x_{0}=0$ without loss of generality, and then show 
\begin{equation}\label{tail.finite}
\int_{\mathbb{R}^{n}\setminus B_{r}}\frac{|f(y)|^{p-1}}{|y|^{p}}\frac{\phi(|y|)}{|y|^{n}}\,dy < \infty.
\end{equation}
Observe that $t \mapsto \phi(t)/t^{n+p}$ is almost decreasing with constant $L$, and that $|x-y| \le 2|y|$ for any $x \in B_{r}$ and $y\in\mathbb{R}^{n}\setminus B_{r}$. We thus have
\begin{equation*}
\begin{aligned}
\infty & > \int_{\Omega}|f(x)|^{p}\,dx + \iint_{\mathcal{C}_{\Omega}}\frac{|f(x)-f(y)|^{p}}{|x-y|^{p}}\frac{\phi(|x-y|)}{|x-y|^{n}}\,dx\,dy \\
& \ge \int_{B_{r}}|f(x)|^{p}\,dx + L^{-1}\int_{\mathbb{R}^{n}\setminus B_{r}}\int_{B_{r}}\frac{|f(x)-f(y)|^{p}}{|y|^{p}}\frac{\phi(|y|)}{|y|^{n}}\,dx\,dy.
\end{aligned}
\end{equation*}
From the above display and the inequality $|f(x)|^{p} + |f(x)-f(y)|^{p} \ge 2^{-p+1}|f(y)|^{p}$, it follows that
\begin{equation}\label{abc}
\begin{aligned}
\infty & > \int_{\mathbb{R}^{n}\setminus B_{r}}\int_{B_{r}}\frac{|f(x)|^{p}}{|y|^{p}}\frac{\phi(|y|)}{|y|^{n}}\,dx\,dy + \int_{\mathbb{R}^{n}\setminus B_{r}}\int_{B_{r}}\frac{|f(x)-f(y)|^{p}}{|y|^{p}}\frac{\phi(|y|)}{|y|^{n}}\,dx\,dy \\
& \ge 2^{-p+1}\int_{\mathbb{R}^{n}\setminus B_{r}}\int_{B_{r}}\frac{|f(y)|^{p}}{|y|^{p}}\frac{\phi(|y|)}{|y|^{n}}\,dx\,dy = 2^{-p+1}|B_{r}|\int_{\mathbb{R}^{n}\setminus B_{r}}\frac{|f(y)|^{p}}{|y|^{p}}\frac{\phi(|y|)}{|y|^{n}}\,dy,
\end{aligned}
\end{equation}
where we have also used \eqref{out.int}. 
Using \eqref{out.int} once again, we apply H\"older's inequality to the last integral in \eqref{abc}, which leads to \eqref{tail.finite}. 
\end{proof}

\subsection{Existence of weak solutions}\label{subsec3}
In this section, we prove the existence and uniqueness of weak solutions to the Dirichlet problem
\begin{equation}\label{Dirichlet}
\left\{
\begin{aligned}
\mathcal{L}u &= 0&\text{in }& \Omega, \\
u &= g&\text{in }& \mathbb{R}^{n}\setminus\Omega,
\end{aligned}
\right.
\end{equation}
where $g \in \mathbb{W}^{\phi,p}(\Omega)$ is a given boundary datum. By a standard argument (see for instance \cite[Theorem~2.3]{DKP16}), $u \in \mathbb{W}^{\phi,p}(\Omega)$ is a weak solution to \eqref{mainPDE} if and only if it is a minimizer of
\begin{equation}\label{ftl}
\mathcal{F}(w;\Omega) \coloneqq \iint_{\mathcal{C}_{\Omega}}|w(x)-w(y)|^{p}K(x,y)\,dx\,dy,
\end{equation}
which means that 
\[ \mathcal{F}(u;\Omega) \le \mathcal{F}(w;\Omega) \]
for any $w \in \mathbb{W}^{\phi,p}(\Omega)$ with $w=u$ a.e. in $\mathbb{R}^{n}\setminus \Omega$. 
In this point of view, we aim to prove existence and uniqueness of minimizers of \eqref{ftl} over the convex admissible set
\[ \mathbb{W}^{\phi,p}_{g}(\Omega) \coloneqq \{w\in\mathbb{W}^{\phi,p}(\Omega): w=g \ \text{ a.e. in }\mathbb{R}^{n}\setminus\Omega\} \]
via direct methods of the calculus of variations.

\begin{theorem}
Let $\phi:[0,\infty)\rightarrow[0,\infty)$ be a measurable function satisfying \eqref{adec} for some $s \in (0,1)$. Under assumption \eqref{kernel}, there exists a unique minimizer $u \in \mathbb{W}^{\phi,p}_{g}(\Omega)$ of \eqref{ftl}, which is the weak solution to \eqref{Dirichlet}.
\end{theorem}
\begin{proof}
The uniqueness follows from the strict convexity of the functional \eqref{ftl}, so we prove the existence only. 
Observe that the admissible set $\mathbb{W}^{\phi,p}_{g}(\Omega)$ is nonempty, since $g \in \mathbb{W}^{\phi,p}_{g}(\Omega)$. 
Let $\{u_{j}\} \subset \mathbb{W}^{\phi,p}_{g}(\Omega)$ be a minimizing sequence of \eqref{ftl}, i.e.,
\[ \lim_{j\to\infty}\mathcal{F}(u_{j};\Omega) = \inf_{w\in\mathbb{W}^{\phi,p}_{g}}\mathcal{F}(w;\Omega) < \infty. \]
Then there exists a constant $M>0$ satisfying
\[ \iint_{\mathcal{C}_{\Omega}}\frac{|u_{j}(x)-u_{j}(y)|^{p}}{|x-y|^{p}}\frac{\phi(|x-y|)}{|x-y|^{n}}\,dx\,dy \le \Lambda\mathcal{F}(u_{j};\Omega) \le M \qquad \text{for any}\;\; j \in \mathbb{N}. \]
Now, with $R \coloneqq \diam (\Omega)$, we choose a ball $B_{R} \equiv B_{R}(z)$ such that $\Omega \subset B_{R}$. 
Let $v_{j} \coloneqq u_{j}-g$, then $v_{j} \in \mathbb{W}^{\phi,p}(\Omega)$ and $v_{j}=0$ a.e. in $\mathbb{R}^{n}\setminus \Omega$. 
By Proposition~\ref{prop:embedding} (iii),
it satisfies
\begin{equation*}
\begin{aligned}
& \frac{\phi(4R)}{L(4R)^{(1-s)p}}\int_{B_{2R}}\int_{B_{2R}}\frac{|v_{j}(x)-v_{j}(y)|^{p}}{|x-y|^{s p}}\frac{dx\,dy}{|x-y|^{n}} \le \int_{B_{2R}}\int_{B_{2R}}\frac{|v_{j}(x)-v_{j}(y)|^{p}}{|x-y|^{p}}\frac{\phi(|x-y|)}{|x-y|^{n}}\,dx\,dy \\
& \le c\iint_{\mathcal{C}_{\Omega}}\frac{|u_{j}(x)-u_{j}(y)|^{p}}{|x-y|^{p}}\frac{\phi(|x-y|)}{|x-y|^{n}}\,dx\,dy + c\iint_{\mathcal{C}_{\Omega}}\frac{|g(x)-g(y)|^{p}}{|x-y|^{p}}\frac{\phi(|x-y|)}{|x-y|^{n}}\,dx\,dy \\
& \le c\left(M+\iint_{\mathcal{C}_{\Omega}}\frac{|g(x)-g(y)|^{p}}{|x-y|^{p}}\frac{\phi(|x-y|)}{|x-y|^{n}}\,dx\,dy\right)
\end{aligned}
\end{equation*}
for any $j \in \mathbb{N}$. In particular, this along with the fractional Poincar\'e inequality \cite[Lemma~4.7]{Coz17} implies that $\{v_{j}\}$ is bounded in $W^{s,p}(B_{R})$. 
Then we apply the compact embedding of fractional Sobolev spaces \cite[Theorem~7.1]{DPV} to see that, up to non-relabeled subsequences, there exists a function $v \in W^{s,p}(B_{R})$ such that
\begin{equation*}
\left\{
\begin{aligned}
v_{j} \rightharpoonup v & \quad \text{in}\;\; W^{s,p}(B_{R})  \\
v_{j} \rightarrow v & \quad \text{in}\;\; L^{p}(B_{R}) \\
v_{j} \rightarrow v & \quad \text{a.e. in}\;\; B_{R}.
\end{aligned}
\right.
\end{equation*}
We further extend $v$ to $\mathbb{R}^{n}$ by letting $v=0$ on $\mathbb{R}^{n}\setminus B_{R}$, and then define $u \coloneqq v+g$. Then $u=g$ a.e. in $\mathbb{R}^{n}\setminus\Omega$ and $u_{j}\rightarrow u$ a.e. in $\mathbb{R}^{n}$, so Fatou's lemma implies
\[ \mathcal{F}(u;\Omega) \le \liminf_{j\rightarrow\infty}\mathcal{F}(u_{j};\Omega) = \inf_{w\in\mathbb{W}^{\phi,p}_{g}(\Omega)}\mathcal{F}(w;\Omega). \]
This shows that $u \in\mathbb{W}^{\phi,p}_{g}(\Omega)$ and it is a minimizer of \eqref{ftl}, as desired. 
\end{proof}

\subsection{Iteration lemmas}
We finally recall two standard iteration lemmas, see for instance \cite{Giu}.

\begin{lemma}\label{lemseq}
Let $\{y_i\}_{i=0}^\infty$ be a sequence of nonnegative numbers and satisfy
\[
y_{i+1}\leq b_1b_2^{i}y_i^{1+\beta}, \ \ \ i=0,1,2,\dots
\]
for some $b_1,\beta>0$ and $b_2>1$. If 
\[
y_0\leq b_1^{-1/\beta}b_2^{-1/\beta^2},
\]
then $y_i\to 0$ as $i\to \infty$.
\end{lemma}

\begin{lemma}\label{tech.lemma}
Let $h:[R,2R] \rightarrow [0,\infty)$ be a bounded function that satisfies
\begin{equation*}
h(r_{1}) \le \vartheta h(r_{2}) + \frac{C_{1}}{(r_{2}-r_{1})^{\nu}} + C_{2}
\end{equation*}
for any $r_{1},r_{2}$ with $R \le r_{1} < r_{2} \le 2R$, where $\vartheta \in (0,1)$ and $C_{1},C_{2},\nu > 0$ are given constants. Then there exists a constant $c = c(\vartheta,\nu)>0$ such that
\begin{equation*}
h(R) \le c\left[\frac{C_{1}}{R^{\nu}} + C_{2}\right].
\end{equation*}
\end{lemma}

\section{Sobolev-Poincar\'e inequalities}\label{sec3}

In this section, we obtain sharp  Sobolev-Poincar\'e type inequalities for $W^{\phi,p}$-functions, which generalize those in \cite{BBM02} concerned with $W^{s,p}$. 
As mentioned in the introduction, a main point here is the stability of the estimates as $s \nearrow 1$, which plays a crucial role in our analysis of equation \eqref{mainPDE}.
To show this, we modify and develop the approach in \cite{BDOR}.

\begin{theorem}\label{thm:SoboPoin}
Let $s \in (0,1)$ and $p\in (1,\infty)$ with $s p < n$, and let $\phi : [0,\infty)\to [0,\infty)$ be a measurable function satisfying \eqref{Dini}, \eqref{adec} and \eqref{ainc}. 
If $v \in W^{\phi,p}(B_r)$ for a ball $B_r \subset \R^n$, then
\begin{equation}\label{Sobopoin}
\left(\mean{B_r} |v-(v)_{B_r}|^{p^{*}_{s}}\,dx\right)^{p/p^{*}_{s}} \le \frac{c}{s^p(n-s p)^{p-1}} \frac{r^p}{\Phi(r)}\mean{B_r} \int_{B_r} \frac{|v(x)-v(y)|^p}{|x-y|^p} \frac{\phi(|x-y|)}{|x-y|^n}\,dy\,dx 
\end{equation}
holds for a constant $c=c(n,p,\tilde{s},L)>0$, where  $p_s^* \coloneqq np/(n-s p)$. 
\end{theorem}
\begin{proof}
In this proof, all the implicit constants and constants $c$ depend only on $n$, $p$, $\tilde{s}$ and $L$, but not on $s$. We may assume that the center of $B_r$ is the origin, and set $\rho \coloneqq r/2$ and $B_\rho=B_\rho(0)$.

\textit{Step 1: Reduction to differentiable function. } We first find
a function that is equivalent to $\phi$, but satisfies the same condition as $\phi$ and as well as additional conditions. 
%Since $t\phi(t)/t^{1+(1-s)p}$ is almost decreasing with constant $L$ for $t\in(0,\infty)$, by the same argument as in \eqref{phi.Phi},
%\[
%t\phi(t) \ge  \int_0^t \phi(\tau)\, d\tau  \ge \frac{t\phi(t)}{L(1+(1-s)p)}   \ge  \frac{t\phi(t)}{L(1+p)}
%\quad \text{for all }\ t>0.
%\]
%
%\,
Define $\tilde{\phi}(t) \coloneqq \phi(t)/t^{(1-\tilde{s})p}$, then $\tilde{\phi}$ is almost increasing and $t\mapsto t\tilde{\phi}(t)/t^{1+(\tilde{s}-s)p}$ is almost decreasing.

\[
Lt\tilde{\phi}(t) \ge  \int_0^t \tilde{\phi}(\tau)\, d\tau  \ge \frac{t\tilde{\phi}(t)}{L(1+(\tilde{s}-s)p)}   \ge  \frac{t\tilde{\phi}(t)}{L(1+\tilde{s}p)}
\quad \text{for all }\ t>0.
\]
%Define $\psi:[0,\infty)\to [0,\infty)$ by $\psi(0) \coloneqq 0$ and 
%\[
% \psi(t) \coloneqq \frac{1}{t}\int_0^t \phi(\tau)\, d\tau  
%\quad \text{for }\ t>0.
%\]
%
%\,
Define $\tilde{\psi},\psi:[0,\infty) \to [0,\infty)$ by $\tilde{\psi}(0) \coloneqq 0$, $\psi(0) \coloneqq 0$ and
\[ 
\tilde{ \psi}(t) \coloneqq \frac{1}{t}\int_0^t \tilde{\phi}(\tau)\, d\tau \quad \text{and} \quad \psi(t) \coloneqq t^{(1-\tilde{s})p}\tilde{\psi}(t)  
\quad \text{for }\ t>0.
\]
%\,
%
%Then, $\psi$  is nondecreasing and differentiable on $(0,\infty)$, and satisfies the inequality
%\[
%\{L(1+p)\}^{-1} \phi(t) \le \psi(t) \le \phi(t),
%\quad \text{for all  }\ t\ge 0,
%\]
Then $\tilde{\psi}$  is nondecreasing and differentiable on $(0,\infty)$. Moreover, $\psi$ satisfies the inequality
\[
\{L(1+\tilde{s}p)\}^{-1}\phi(t) \le \psi(t) \le \phi(t)
\quad \text{for all  }\ t\ge 0,
\]
\eqref{Dini} and \eqref{adec} with  $\phi$ replaced by  $\psi$ and with the same $p$, $s$ and $\tilde{s}$ as for $\phi$. In addition, the function $t\mapsto t \tilde{\psi}(t)$, $t\in[0,\infty)$, is nondecreasing, convex and satisfies the $\Delta_2$ condition, i.e., $(2t)\tilde{\psi}(2t)\le c(p,\tilde{s},L) t\tilde{\psi}(t)$ for all $t\ge 0$. Consequently, by \cite[Lemma 2.2.6]{HH19book}, there exists a large constant $\tilde{q}>0$ depending only on $\tilde{s}$, $p$ and $L$ such that $t \mapsto t\tilde{\psi}(t)/t^{\tilde{q}+1}= \psi(t)/t^{\tilde{q}+(1-\tilde{s})p}$ is nonincreasing on $(0,\infty)$. 
Therefore, without loss of generality, we may additionally assume that $\phi$ is differentiable on $(0,\infty)$, and that
\begin{equation}\label{decq}
\frac{\phi(t_2)}{t_2^q} \le \frac{\phi(t_1)}{t_1^q}
\quad \text{for all }\ 0<t_1<t_2<\infty, 
\end{equation}
where $q = q(p,s,\tilde{s},L)$ is a fixed constant, in particular satisfying $q>pL$ 
(see \eqref{crhorange} below).

\textit{Step 2: A Riesz-type potential estimate. } 
Define 
\begin{equation}\label{def:eta}
\tilde \eta(x) \coloneqq \left(\frac{\phi(|x|)}{|x|^n} - \frac{|x|^{q-n}\phi(\rho)}{\rho^{q}} \right) \chi_{B_\rho}(x)
\quad \text{and}\quad
 \eta(x) \coloneqq \frac{c_\rho}{\Phi(\rho) \omega_n} \tilde \eta(x),
\end{equation}
where  $c_\rho \coloneqq  \Phi(\rho)\omega_n /\|\tilde \eta\|_{L^1(B_\rho)}$.
Note that $\eta,\tilde\eta\ge 0$ by \eqref{decq}, $\eta \in W^{1,1}(\R^n \setminus \{0\})$ and
$\|\eta\|_{L^1(\R^n)} =  \|\eta\|_{L^1(B_\rho)} = 1$.
Moreover, 
\[
\int_{B_\rho}\tilde\eta(x)\,dx =\omega_n \Phi(\rho) -  \omega_n\frac{\phi(\rho)}{\rho^{q}} \int^\rho_0 \tau^{q-1}\, d\tau =\omega_n \Phi(\rho) -  \tfrac{1}{q}\omega_n \phi(\rho)\,,
\]
which together with \eqref{phi.Phi} implies that
\[
 \left(1-\tfrac{pL}{q}\right)\omega_n \Phi(\rho) \le  \int_{B_\rho}\tilde\eta(x)\,dx \le \omega_n\Phi(\rho)
\]
and hence 
\begin{equation}\label{crhorange}
c_\rho  \in \left[1,\frac{q}{q-pL}\right].
\end{equation}
Denote by $\eta_t(x) \coloneqq t^{-n} \eta(x/t)$ for $t>0$. Then $\|\eta_t\|_{L^{1}(\R^n)}= \| \eta_t \|_{L^{1}(B_{t\rho})} =1$ and $\eta_1=\eta$. 

In this step, we show that for every Lebesgue point $x\in B_{r}=B_{2\rho}$  of $v$,
\begin{equation} \label{eq:Rieszpotential}
\begin{aligned}
&\bigg|v(x) - \mean{B_{\rho}} (v* \eta)(y)\,dy\bigg| \le \mean{B_{\rho}} |v(x) - (v* \eta)(y)|\,dy \\
&\hspace{2cm}  \le c r^{1-s} \int_{B_r}\left[\frac{1}{\Phi(r)}\int_{B_r }   \frac{|v(\xi) - v(z)|^{p}}{|\xi-z|^{p}}\frac{\phi(|\xi-z|)}{|\xi-z|^{n}}\,d\xi \right]^{1/p} \frac{1}{|x-z|^{n-s}} \,dz.
\end{aligned}\end{equation}
We split the integral into
\begin{equation*}
\begin{aligned}
\mean{B_{\rho}} |v(x) - (v* \eta)(y)|\,dy
&\leq \mean{B_{\rho}}|v(x) - (v* \eta_{\frac{|x-y|}{4\rho}})(y)|\,dy  +\mean{B_{\rho}}|v(y) - (v* \eta_{\frac{|x-y|}{4\rho}})(y)|\,dy \\
&\quad\; +\mean{B_{\rho}}|v(y) - (v* \eta)(y)|\,dy\\
& \eqqcolon I_{1} + I_{2} + I_{3}.
\end{aligned}
\end{equation*}
We start by estimating $\mathrm{I}$. Since $x$ is a Lebesgue point of $v$, we can calculate
\begin{equation}\label{I.1st}
\begin{aligned}
& v(x) - (v*\eta_{\frac{|x-y|}{4\rho}})(y) \\
&= -\int_0^1 \frac{\partial}{\partial t} \left[(v * \eta_{t\frac{|x-y|}{4\rho}})\big(x + t(y-x)\big)\right] \,dt \\
&= \int_0^1 \frac{\partial }{\partial t} \left[\int_{\R^n} \big\{v\big(x+t(y-x)\big)- v(z)  \big\} \eta_{t\frac{|x-y|}{4\rho}}\big(x+t(y-x)-z\big)  \,dz\right] \,dt \\
&= \int_0^1  \int_{B_{2\rho}} \big\{v\big(x+t(y-x)\big) - v(z) \big\} \frac{\partial}{\partial t}  \left[\eta_{t\frac{|x-y|}{4\rho}}\big(x+t(y-x)-z\big)\right]  \,dz \,dt,
\end{aligned}
\end{equation}
where in the second step we have used the facts that 
\[  \int_{\R^n} \frac{\partial}{\partial t} \left[\eta_{t\frac{|x-y|}{4\rho}}(x+t(y-x)-z)\right] \,dz = \frac{\partial}{\partial t} \int_{\R^n}  \eta_{t\frac{|x-y|}{4\rho}}(x+t(y-x)-z) \,dz= \frac{\partial}{\partial t} 1 =0 \] 
and that $\mathrm{supp}\,  \eta_{t\frac{|x-y|}{4\rho}}(x+t(y-x)-\cdot) \subset B_{2\rho}$ for any $x\in B_{2\rho}$ and $y\in B_{\rho}$, since 
\[ |x+t(y-x)-z|\le t\tfrac{|x-y|}{4} \;\; \Rightarrow \;\; |z| \le t|y| + (1-t)|x|+ t\tfrac{|x-y|}{4} \le t\rho + (1-t)2\rho + t \tfrac{3}{4}\rho\le 2\rho. \]
Now, a direct calculation gives
\begin{equation*}
\begin{aligned}
&\frac{\partial}{\partial t} \left[ \eta_{t \frac{|x-y|}{4\rho}}\big(x+t(y-x)-z\big) \right] \\
& = \frac{\partial}{\partial t}  \left[ \left(\frac{|x-y|}{4\rho}t\right)^{-n} \eta\left(4\rho\frac{x+t(y-x)-z}{t |x-y|}\right) \right] \\
& = -\frac{n}{t} \eta_{t \frac{|x-y|}{4\rho}}\big(x+t(y-x)-z\big) - \left( \frac{|x-y|}{4\rho} t \right)^{-n} D\eta \left(4\rho \frac{x+t(y-x)-z}{t|x-y|}\right) \cdot \frac{4\rho (x-z)}{|x-y|t^2}.
\end{aligned}
\end{equation*}
Note that by \eqref{crhorange} and \eqref{decq} with $|x|< 2\rho$,
\[\begin{split} 
|D\eta(x)| 
&=\frac{c_\rho}{\omega_n \Phi(\rho)} \left| \frac{\phi'(|x|)|x|-n\phi(|x|)}{|x|^{n+1}}- (q-n) \frac{|x|^{q-n-1}\phi(\rho)}{\rho^{q}} \right|  \chi_{B_\rho}(x) \\
&\lesssim \frac{1}{\Phi(\rho)}\left\{\frac{\phi(|x|)}{|x|^{n+1}}+\frac{|x|^{q-n-1}\phi(2\rho)}{(2\rho)^{q}} \right\} \chi_{B_\rho}(x) \\
&\lesssim \frac{1}{\Phi(\rho)}\frac{\phi(|x|)}{|x|^{n+1}} \chi_{B_\rho}(x),
\end{split}\] 
where we used that fact that $\phi'(t) t \le q \phi(t)$ by \eqref{decq}.
Thus, writing $[x,y]_t  \coloneqq x+t(y-x)$, we estimate
\begin{equation*}
\begin{aligned}
\left|\frac{\partial }{\partial t}  \left[\eta_{t\frac{|x-y|}{4\rho}}\big([x,y]_t-z\big)\right] \right| 
&\lesssim \frac{1}{\Phi(\rho)} \chi_{\left\{|[x,y]_t-z| \leq t \frac{|x-y|}{4}\right\}} \Bigg\{ t^{-1} |[x,y]_t-z|^{-n} \phi\left(4\rho \frac{|[x,y]_t-z|}{t|x-y|}\right) \\
&\qquad  \qquad  \qquad  \qquad   + |[x,y]_t-z|^{-n-1} \phi\left(4\rho \frac{|[x,y]_t-z|}{t|x-y|}\right)\frac{ |x-z|}{t} \Bigg\}.
\end{aligned}
\end{equation*}
Note that if $|[x,y]_t-z| \leq  t \frac{|x-y|}{4}$, then
\begin{align*}
|x-z| \geq |x-[x,y]_t|-|[x,y]_t-z|=t |x-y| - |[x,y]_t-z|\geq 3|[x,y]_t-z|.
\end{align*}
This and the previous estimate imply
\begin{align*}
\left|\frac{\partial }{\partial t}  \eta_{t\frac{|x-y|}{4}}\big([x,y]_t-z\big)\right| \lesssim \frac{1}{\Phi(\rho)}\chi_{\left\{|[x,y]_t-z| \leq t \frac{|x-y|}{4}\right\}} \frac{|x-z|}{t |[x,y]_t-z|^{n+1}} \phi\left(4\rho \frac{|[x,y]_t-z|}{t|x-y|}\right).
\end{align*}
Plugging this into \eqref{I.1st}, we obtain
\begin{equation*}
\begin{aligned}
& |v(x) - (v*\eta_{\frac{|x-y|}{4\rho}})(y)|  \\
&\lesssim \frac{1}{\Phi(\rho)}\int_0^1\int_{B_{2\rho}}\chi_{\left\{|[x,y]_t-z| \leq t \frac{|x-y|}{4}\right\}}\frac{ |v([x,y]_t)-v(z)|\, |x-z|}{t |[x,y]_t-z|^{n+1}} \phi\left(4\rho \frac{|[x,y]_t-z|}{t|x-y|}\right)\,dz\,dt .
\end{aligned}
\end{equation*}
Hence, using additionally Fubini's theorem, we arrive at
\begin{equation*}
I_{1} \lesssim \frac{1}{\Phi(\rho)}\mean{B_{2\rho}}  \int_0^1 \int_{B_{\rho}} \chi_{\left\{|[x,y]_t-z| \leq t \frac{|x-y|}{4}\right\}} \frac{|v([x,y]_t)-v(z)|\, |x-z|}{t|[x,y]_t-z|^{n+1}} \phi\left(4\rho\frac{|[x,y]_t-z|}{t|x-y|}\right) \,dy \,dt\, dz.
\end{equation*}
We now substitute $\xi= x+t(y-x)=[x,y]_t$ in the above integral. Then $dy=t^{-n}d\xi$ and $t=\frac{|\xi-x|}{|y-x|}\geq \frac{|\xi-x|}{3\rho}$. Using the fact that
 \begin{equation}\label{equiv0}
|\xi-z| \leq \tfrac{1}{4}  |\xi -x|\ \  \Longrightarrow \ \  \tfrac{3}{4} |x-\xi|\leq |x-z| \leq \tfrac{5}{4} |x-\xi|
\end{equation}
and \eqref{ainc}, we get
\begin{equation*}
\begin{aligned}
I_{1} &\lesssim \frac{1}{\Phi(\rho)}\mean{B_{2\rho}}\int_{B_{2\rho}}\int_{\frac{|\xi-x|}{3\rho}}^1 \chi_{\left\{|\xi-z| \leq \frac{|\xi -x|}{4}\right\}}  t^{-n -1}\frac{|v(\xi)-v(z)|\, |x-z|}{|\xi-z|^{n+1}}  \phi\left(4\rho\frac{|\xi-z|}{|\xi-x|}\right) \,dt \, d\xi\,dz \\
&\lesssim \frac{1}{\Phi(\rho)}\int_{B_{2\rho}}\int_{B_{2\rho}} \chi_{\left\{|\xi-z| \leq \frac{|\xi -x|}{4}\right\}}  \frac{|v(\xi)-v(z)|\,|x-z|}{|\xi-z|^{n+1}|\xi-x|^n}  \phi \left(4\rho \frac{|\xi-z|}{|\xi-x|}\right) \, d \xi\,dz \\
&\lesssim \frac{1}{\Phi(\rho)}\int_{B_{2\rho}}\int_{B_{3\rho}} \chi_{\{|\xi|< 2\rho\}} \chi_{\left\{|\xi-z| \leq \frac{|x-z|}{3}\right\}}  \frac{|v(\xi)-v(z)|\,|x-z|}{|\xi-z|^{n+1}|\xi-x|^n}  \phi \left(5\rho \frac{|\xi-z|}{|x-z|}\right) \, d \xi\,dz .
\end{aligned}
\end{equation*}
At this stage, we observe that, by substituting $\tilde\xi= \frac{5\rho}{|x-z|}(\xi-z)$ and using \eqref{adec},
\begin{equation*}
\begin{aligned}
\int_{B_{3\rho}}  \chi_{\left\{|\xi-z|  \leq \frac{|x-z|}{3}\right\}} \frac{1}{|\xi-z|^{n}}  \phi \left(5\rho \frac{|\xi-z|}{|x-z|}\right) \, d \xi
 & =  \int_{B_{5\rho/3}} \chi_{\left\{\left|\frac{|x-z|}{5\rho}\tilde \xi +z \right| < 3\rho\right\}}   \frac{\phi(|\tilde \xi|)}{|\tilde \xi|^{n}} \, d \tilde \xi \\
 &\le \omega_n \Phi\big(\tfrac 53 \rho \big)\le \omega_n \big(\tfrac53\big)^{(1-s)p} \Phi(\rho)
\end{aligned}
\end{equation*}
whenever $x\in B_{2\rho}$ and $z\in B_{2\rho}$. Moreover, since $z\in B_{2\rho}$ and $|x-z|\le 4\rho$, we have  
\[
\left\{\tilde\xi\in B_{5\rho/3}:\left|\tfrac{|x-z|}{5\rho}\tilde \xi +z \right| < 3\rho\right\} \supset \left\{\tilde\xi\in B_{5\rho/3}:\tfrac{|x-z|}{5\rho}|\tilde \xi| < \rho\right\}\supset  B_{\rho}
\]
and hence 
\[
\int_{B_{3\rho}}  \chi_{\left\{|\xi-z|  \leq \frac{|x-z|}{3}\right\}} \frac{1}{|\xi-z|^{n}}  \phi \left(5\rho \frac{|\xi-z|}{|x-z|}\right) \, d \xi \ge  \int_{B_{\rho}} \frac{\phi(|\tilde \xi|)}{|\tilde \xi|^{n}} \, d \tilde \xi = \omega_n \Phi(\rho).
\]
Summarizing, we have
\begin{equation}\label{phi.measure}
\frac{1}{\Phi(\rho)}\int_{B_{3\rho}}  \chi_{\left\{|\xi-z|  \leq \frac{|x-z|}{3}\right\}} \frac{1}{|\xi-z|^{n}}  \phi \left(5\rho \frac{|\xi-z|}{|x-z|}\right) \, d \xi \approx 1.
\end{equation}
Therefore, by Jensen's inequality, we obtain
\[\begin{aligned}
I_{1} &\lesssim  \int_{B_{2\rho}}\left[ \frac{1}{\Phi(\rho)} \int_{B_{3\rho}}    \chi_{\{|\xi| < 2\rho\}}^p \frac{|v(\xi)-v(z)|^p|x-z|^p}{|\xi-z|^{p}|\xi-x|^{pn}} \frac{\chi_{\{|\xi-z| \leq |x-z|/3\}} }{|\xi-z|^{n}}  \phi \left(5\rho \frac{|\xi-z|}{|x-z|}\right) \, d \xi \right]^{1/p}\, dz\\
&=  \int_{B_{2\rho}}\left[ \frac{1}{\Phi(\rho)} \int_{B_{2\rho}}  \chi_{\left\{|\xi-z| \leq \frac{|x-z|}{3}\right\}}  \frac{|v(\xi)-v(z)|^p|x-z|^p}{|\xi-z|^{n+p}|\xi-x|^{pn}}   \phi \left(5\rho \frac{|\xi-z|}{|x-z|}\right) \, d \xi \right]^{1/p}\, dz.
\end{aligned}\]
Since $5\rho/|x-z| \ge 1$ for $x,z\in B_{2\rho}=B_r$, we use \eqref{adec} and the fact that
\begin{equation}\label{equiv1}
|\xi-z| \le \tfrac{1}{3}|x-z| \;\; \Longrightarrow \;\; \tfrac{2}{3}|x-z| \le |x-\xi| \le \tfrac{4}{3}|x-z|,
\end{equation}
thereby obtaining the following estimate for $I_{1}$:
\begin{equation*}
I_{1}  \lesssim \rho^{1-s} \int_{B_{2\rho}}\left[ \frac{1}{\Phi(\rho)} \int_{B_{2\rho}}  \frac{|v(\xi)-v(z)|^p}{|\xi-z|^{p}}  \frac{ \phi (|\xi-z|)}{|\xi-z|^n} \, d \xi \right]^{1/p}\frac{1}{|\xi-x|^{n-s}}\, dz.
\end{equation*}
Next, we estimate $I_{2}$. By using \eqref{def:eta} and \eqref{crhorange}, Fubini's theorem and \eqref{equiv0} with $\xi$ replaced by $y$, we have
\begin{equation*}
\begin{aligned}
I_{2}&\le \mean{B_{\rho}} \int_{\R^n}
| v(y)-v(z) | \eta_{ \frac{|x-y|}{4\rho}}(y-z)\,dz\,dy\\
&\lesssim \frac{1}{\Phi(\rho)}\mean{B_{2\rho}} \int_{B_{2\rho}}\chi_{\left\{|y-z|\leq \frac{|x-y|}{4}\right\}}\frac{|v(y)-v(z)|}{ |y-z|^{n}}  \phi \left(4\rho \frac{|y-z|}{|x-y|}\right)\,dz\,dy \\
& \lesssim \frac{1}{\Phi(\rho)}\mean{B_{2\rho}}\int_{B_{3\rho}}\chi_{\{|y|<2\rho\}}\chi_{\left\{|y-z|\leq \frac{|x-z|}{3}\right\}}\frac{|v(y)-v(z)|}{ |y-z|^{n}} \phi \left(5\rho \frac{|y-z|}{|x-z|}\right)\,dy\,dz.
\end{aligned}
\end{equation*}
In view of \eqref{phi.measure}, we apply Jensen's inequality to the inner integral in the above display. Moreover, using \eqref{adec} and \eqref{equiv1} with $\xi$ replaced by $y$, 
we estimate $I_{2}$ as
\begin{equation*}
\begin{aligned}
I_{2}&\lesssim  \mean{B_{2\rho}}\left[\frac{1}{\Phi(\rho)}\int_{B_{2\rho}}\chi_{\left\{|y-z|\leq \frac{|x-z|}{3}\right\}}\frac{|v(y)-v(z)|^{p}}{|y-z|^{n}}\phi\left(5\rho\frac{|y-z|}{|x-z|}\right)\,dy\right]^{1/p}\,dz \\
& \lesssim \rho^{1-s}\mean{B_{2\rho}}\left[\frac{1}{\Phi(\rho)}\int_{B_{2\rho}}\chi_{\left\{|y-z|\leq\frac{|x-z|}{3}\right\}}\frac{|v(y)-v(z)|^{p}}{|x-z|^{p(1-s)}}\frac{\phi(|y-z|)}{|y-z|^{n}}\,dy\right]^{1/p}\,dz \\
&\lesssim \rho^{1-s} \int_{B_{2\rho}}\left[\frac{1}{\Phi(\rho)}\int_{B_{2\rho}} \frac{|v(y)-v(z)|^{p}}{|y-z|^{p}} \frac{\phi(|y-z|)}{|y-z|^{n}}\,dy\right]^{1/p}\frac{1}{|x-z|^{n-s}}\,dz.
\end{aligned}
\end{equation*}
Finally, we estimate $I_{3}$. Here we use \eqref{def:eta}, the inequality $|x-y|+|y-z|<8\rho$ for any $x,y,z \in B_{2\rho}$ and Jensen's inequality in order to have
\begin{equation*}
\begin{aligned}
I_{3}&\le \mean{B_{\rho}} \int_{\R^n} |v(y)-v(z)| \eta(y-z)\,dz\,dy\\
&\lesssim \frac{1}{\Phi(\rho)} \mean{B_{2\rho}}\int_{B_{2\rho}}|v(y)-v(z)| \frac{\phi(|y-z|)}{|y-z|^{n}}\,dz\,dy \\
&\lesssim \frac{\rho^{1-s}}{\Phi(\rho)} \int_{B_{2\rho}}\left[\int_{B_{2\rho}}\frac{|v(y)-v(z)|}{|y-z|} \frac{\phi(|y-z|)}{|y-z|^{n}}\,dz\right]\frac{1}{|x-y|^{n-s}}\,dy \\
&\lesssim \rho^{1-s} \int_{B_{2\rho}} \left[\frac{1}{\Phi(\rho)}\int_{B_{2\rho}} \frac{|v(y)-v(z)|^{p}}{|y-z|^p}\frac{\phi(|y-z|) }{|y-z|^{n}}\,dz\right]^{1/p}\frac{1}{|x-y|^{n-s}}\,dy.
\end{aligned}
\end{equation*}
Therefore, combining the estimates found for $I_{1}$, $I_{2}$ and $I_{3}$, we obtain  \eqref{eq:Rieszpotential}.

\textit{Step 3: Proof of \eqref{Sobopoin}. } 
We recall the following Hardy-Littlewood-Sobolev inequality \cite{HarLi28,Hedberg,Sobolev}:
\[
\left(\int_{\mathbb{R}^{n}} \left[\int_{\mathbb{R}^n} \frac{|f(y)|}{|y-x|^{n-s}}\,dy \right]^{p^{*}_{s}}\,dx \right)^{1/p^{*}_{s}} \le \frac{c}{s(n-s p)^{1-1/p}} \left(\int_{\mathbb{R}^{n}} |f(x)|^{p} \,dx\right)^{1/p}
\]
for any $f \in L^{p}(\mathbb{R}^{n})$ with $s p <n $. 
Applying this inequality to
\begin{equation*}
f(y)=\left(\int_{B_r } \frac{|v(y) - v(z)|^p}{|y-z|^{p}}\frac{\phi(|y-z|)}{|y-z|^n}\, dz\right)^{1/p}\chi_{B_{r}}(y),
\end{equation*}
and then combining the resulting estimate with  \eqref{eq:Rieszpotential}, we have 
\begin{equation*}
\begin{aligned}
&\left(\mean{B_r} |v(x)-(v)_{B_r}|^{p^{*}_{s}}\, dx\right)^{p/p^{*}_{s}} \\
& \le c\left(\mean{B_{r}}\bigg|v(x)-\mean{B_{\rho}}(v*\eta)(y)\,dy\bigg|^{p^{*}_{s}}\,dx\right)^{p/p^{*}_{s}} \\
&\le c\frac{r^{(1-s)p}}{\Phi(r)}\left\{\mean{B_r} \left[\int_{B_{r}}\left( \int_{B_r }   \frac{|v(y) - v(z)|^{p}}{|y-z|^{p}}\frac{  \phi(|y-z|)}{|y-z|^n}\, dz\right)^{1/p} \frac{1}{|y-x|^{n-s}}\,dy \right]^{p^{*}_{s}}\,dx \right\}^{p/p^{*}_{s}}  \\
&\le \frac{c}{s^{p}(n-s p)^{p-1}}\frac{ r^{p}}{\Phi(r)}  \mean{B_r}  \int_{B_r }   \frac{|v(x) - v(z)|^p}{|x-z|^p} \frac{\phi(|x-z|)}{|x-z|^n} \,  dz \, dx .
\end{aligned}
\end{equation*}
This implies \eqref{Sobopoin}. 
\end{proof}

\begin{remark}
From Step 1 in the proof of the above theorem, 
we can also obtain a Riesz-type potential estimate for $\big|v(x) - (v)_{B_r}\big|$ for $x\in B_r$. We first observe by \eqref{eq:Rieszpotential} and Fubini's theorem that for every $x\in B_r$,
\begin{equation*}\begin{aligned}
&\mean{B_{r}} \bigg|v(z) - \mean{B_{\rho}} (v* \eta)(y)\,dy\bigg| \, dz \\
&\le c r^{1-s}\mean{B_{r}} \int_{B_r}\left[\frac{1}{\Phi(r)}\int_{B_r}   \frac{|v(\xi) - v(y)|^{p}}{|\xi-y|^{p}}\frac{\phi(|\xi-y|)}{|\xi-y|^{n}}\,d\xi \right]^{1/p} \frac{1}{|z-y|^{n-s}} \,dy \, dz\\
&\le \frac{c r}{s}\mean{B_{r}} \left[\frac{1}{\Phi(r)}\int_{B_r}   \frac{|v(\xi) - v(y)|^{p}}{|\xi-y|^{p}}\frac{\phi(|\xi-y|)}{|\xi-y|^{n}}\,d\xi \right]^{1/p}  \,dy\\
&\le \frac{c r^{1-s}}{s}\int_{B_{r}} \left[\frac{1}{\Phi(r)}\int_{B_r}   \frac{|v(\xi) - v(y)|^{p}}{|\xi-y|^{p}}\frac{\phi(|\xi-y|)}{|\xi-y|^{n}}\,d\xi \right]^{1/p} \frac{1}{|x-z|^{n-s}} \,dy,
\end{aligned}\end{equation*}
where we used the fact that $|x-z|^{n-s}\le (2r)^{n-s}$.
Then, again using  \eqref{eq:Rieszpotential}, we obtain for almost every $x\in B_r$ that 
\begin{equation*}
\begin{aligned}
\big|v(x) - (v)_{B_r}\big| & \le \bigg|v(x) - \mean{B_{\rho}} (v* \eta)(y)\,dy\bigg| + \mean{B_{r}} \bigg|v(z) - \mean{B_{\rho}} (v* \eta)(y)\,dy\bigg| \, dz \\
&\le \frac{c r^{1-s}}{s}\int_{B_{r}} \left[\frac{1}{\Phi(r)}\int_{B_r}   \frac{|v(\xi) - v(y)|^{p}}{|\xi-y|^{p}}\frac{\phi(|\xi-y|)}{|\xi-y|^{n}}\,d\xi \right]^{1/p} \frac{1}{|x-y|^{n-s}} \,dy .
\end{aligned}
\end{equation*}
We further note that this last estimate, as well as  \eqref{eq:Rieszpotential}, continues to hold when $s p \ge n$ or $p=1$ (see also the proof of \cite[Lemma~A.1]{BDOR}). In particular, when $p=1$, it gives an analog of the following classical estimate:
\[ \big|v(x)-(v)_{B_{r}}\big| \le c\int_{B_{r}}\frac{|Dv(y)|}{|x-y|^{n-1}}\,dy. \]
\end{remark}

We extend Theorem~\ref{thm:SoboPoin} to the case $s p \ge n$.
\begin{corollary}\label{cor.SP}
Let $s \in (0,1)$ and $p\in (1,\infty)$ with $s p \ge  n$, and let $\phi : [0,\infty)\to [0,\infty)$ be a nondecreasing function that satisfies \eqref{Dini} and \eqref{adec}.  
For every $\tilde p \in (p,\infty)$, we have
\begin{equation*}
\left(\mean{B_r} |v-(v)_{B_r}|^{\tilde p}\,dx\right)^{p/\tilde p} \le c \frac{\tilde p^{2p-1}}{(\tilde p-p)^{p}}\frac{r^p}{\Phi(r)}\mean{B_r} \int_{B_r} \frac{|v(x)-v(y)|^p}{|x-y|^p} \frac{\phi(|x-y|)}{|x-y|^n} \,dy\,dx
\end{equation*}
for any $v\in W^{\phi,p}(B_r)$ with $B_r \subset \mathbb{R}^{n}$, where $c = c(n,p,L)>0$. 
\end{corollary}
\begin{proof}
Choose $0<s_0<n/p\le s$ such that $\tilde p= p_{s_0}^*=np/(n-s_0 p)$. Then, since $\phi$ also satisfies \eqref{adec} with $s$ replaced by $s_0$, by Theorem~\ref{thm:SoboPoin},
\[\begin{aligned}
\left(\mean{B_r} |v-(v)_{B_r}|^{\tilde p}\,dx\right)^{p/\tilde p} 
&=\left(\mean{B_r} |v-(v)_{B_r}|^{p^{*}_{s_0}}\,dx\right)^{p/p^{*}_{s_0}} \\
&\le \frac{c}{s_0^p(n-s_0 p)^{p-1}} \frac{r^p}{\Phi(r)}\mean{B_r} \int_{B_r} \frac{|v(x)-v(y)|^p}{|x-y|^p} \frac{\phi(|x-y|)}{|x-y|^n}\,dy\,dx, 
\end{aligned}\]
which implies the desired estimate.
\end{proof}

\section{Fundamental estimates}\label{sec4}
We start this section with two fundamental estimates for \eqref{mainPDE}. 
The first one is a Caccioppoli estimate with tail.

\begin{lemma}\label{lem.ccp}
Let $u \in \mathbb{W}^{\phi,p}(\Omega)$ be a weak solution to \eqref{mainPDE} under assumptions \eqref{kernel}, \eqref{Dini}, \eqref{adec} and \eqref{ainc}, and let $B_{r}\equiv B_{r}(x_{0}) \subset \Omega$ be a ball. Then, with $w_{\pm} \coloneqq (u-k)_{\pm}$ for any $k\in\mathbb{R}$, we have
\begin{equation}\label{ccp.est}
\begin{aligned}
& \int_{B_{\rho}}\int_{B_{\rho}}\frac{|w_{\pm}(x)-w_{\pm}(y)|^{p}}{|x-y|^{p}}\frac{\phi(|x-y|)}{|x-y|^{n}}\,dx\,dy \\
& \le c\frac{\Phi(r)}{(r-\rho)^{p}}\int_{B_{r}}w_{\pm}^{p}\,dx + c\left(\frac{r}{r-\rho}\right)^{n+\tilde{s}p} \frac{\Phi(r)}{r^{p}}[\tail(w_{\pm};r)]^{p-1}\int_{B_{r}}w_{\pm}\,dx
\end{aligned}
\end{equation}
whenever $\rho \in (0,r)$, where $c = c(n,p,\tilde{s},\Lambda,L)>0$. 
\end{lemma}
\begin{proof}
Let $\eta \in C^{\infty}_{0}(B_{(\rho+r)/2})$ be a cut-off function satisfying $\eta \equiv 1$ in $B_{\rho}$ and $|D\eta| \le 4/(r-\rho)$. 
Testing \eqref{mainPDE} with $\eta^{p}(u-k)_{\pm}$ and following the proof of \cite[Theorem~1.3]{DKP16}, we have
\begin{equation*}
\begin{aligned}
& \int_{B_{r}}\int_{B_{r}}|w_{\pm}(x)\eta(x)-w_{\pm}(y)\eta(y)|^{p}K(x,y)\,dx\,dy \\
& \le c\int_{B_{r}}\int_{B_{r}}(\max\{w_{\pm}(x),w_{\pm}(y)\})^{p}|\eta(x)-\eta(y)|^{p}K(x,y)\,dx\,dy \\
& \quad + c\int_{B_{r}}w_{\pm}(x)\eta^{p}(x)\,dx \left(\sup_{x\in\supp\eta}\int_{\mathbb{R}^{n}\setminus B_{r}}w_{\pm}^{p-1}(y)K(x,y)\,dy\right)
\end{aligned}
\end{equation*}
for a constant $c= c(p)>0$. By symmetry, \eqref{Dini} and the fact that $\Phi(2r) \le 2^{p}L \Phi(r)$, the first term in the right-hand side is estimated as
\begin{equation*}
\begin{aligned}
& \int_{B_{r}}\int_{B_{r}}(\max\{w_{\pm}(x),w_{\pm}(y)\})^{p}|\eta(x)-\eta(y)|^{p}K(x,y)\,dx\,dy \\
& \le c\int_{B_{r}}\int_{B_{r}}w_{\pm}^{p}(x)|\eta(x)-\eta(y)|^{p}\frac{\phi(|x-y|)}{|x-y|^{n+p}}\,dx\,dy \\
& \le \frac{c}{(r-\rho)^{p}}\int_{B_{r}}w_{\pm}^{p}(x)\left(\int_{B_{2r}(x)}\frac{\phi(|x-y|)}{|x-y|^{n}}\,dy\right)\,dx \\
& \le c\frac{\Phi(r)}{(r-\rho)^{p}}\int_{B_{r}}w_{\pm}^{p}\,dx,
\end{aligned}
\end{equation*}
where $c= c(n,p,\Lambda,L)>0$. 
To estimate the second term, we observe the facts that $\phi(t)/t^{n+p}$ is almost decreasing for $t\in(0,\infty)$ with constant $L$ and that
\[ |x-y| \ge |y-x_{0}| - |x-x_{0}| \ge |y-x_{0}| - \frac{r+\rho}{2}\frac{|y-x_{0}|}{r} = \frac{r-\rho}{2r}|y-x_{0}| \]
for any $x \in B_{(\rho+r)/2}$ and $y \in \mathbb{R}^{n}\setminus B_{r}$, which along with \eqref{ainc} imply
\[ \frac{\phi(|x-y|)}{|x-y|^{n+p}} \le \left(\frac{2r}{r-\rho}\right)^{n+p}\frac{L}{|y-x_{0}|^{n+p}}\phi\left(\frac{r-\rho}{2r}|y-x_{0}|\right) \le c\left(\frac{r}{r-\rho}\right)^{n+\tilde{s}p}\frac{\phi(|y-x_{0}|)}{|y-x_{0}|^{n+p}}. \]
We thus have
\begin{equation*}
\begin{aligned}
& \int_{B_{r}}w_{\pm}(x)\eta^{p}(x)\,dx \left(\sup_{x\in\supp\eta}\int_{\mathbb{R}^{n}\setminus B_{r}}w_{\pm}^{p-1}(y)K(x,y)\,dy\right) \\
& \le c\int_{B_{r}}w_{\pm}(x)\eta^{p}(x)\,dx \left(\sup_{x\in\supp\eta}\int_{\mathbb{R}^{n}\setminus B_{r}}w_{\pm}^{p-1}(y)\frac{\phi(|x-y|)}{|x-y|^{n+p}}\,dy\right) \\
& \le c\left(\frac{r}{r-\rho}\right)^{n+\tilde{s}p}\int_{B_{r}}w_{\pm}(x)\,dx \cdot \int_{\mathbb{R}^{n}\setminus B_{r}}w_{\pm}^{p-1}(y)\frac{\phi(|y-x_{0}|)}{|y-x_{0}|^{n+p}}\,dy \\
& = c\left(\frac{r}{r-\rho}\right)^{n+\tilde{s}p}\frac{\Phi(r)}{r^{p}}[\tail(w_{\pm};r)]^{p-1}\int_{B_{r}}w_{\pm}\,dx
\end{aligned}
\end{equation*}
for some $c= c(n,p,\Lambda,L)>0$. Therefore, we conclude with \eqref{ccp.est}.
\end{proof}

\begin{remark}
Estimate \eqref{ccp.est} continues to hold for $w_{+}$ (resp. $w_{-}$) if $u$ is merely a weak subsolution (resp. supersolution) to \eqref{mainPDE}.
\end{remark}

The second estimate is a logarithmic estimate. 
\begin{lemma}\label{lem.log}
Let $u \in \mathbb{W}^{\phi,p}(\Omega)$ be a weak supersolution to \eqref{mainPDE} under assumptions \eqref{kernel}, \eqref{Dini}, \eqref{adec} and \eqref{ainc}, which is nonnegative in a ball $B_{R} \equiv B_{R}(x_{0})\subset \Omega$. Then for any $d>0$ and $r \in (0,R/2]$, we have
\begin{equation}\label{log.est}
\begin{aligned}
&  \int_{B_{r}}\int_{B_{r}}\frac{|\log(u(x)+d)-\log(u(y)+d)|^{p}}{|x-y|^{p}}\frac{\phi(|x-y|)}{|x-y|^{n}}\,dx\,dy \\
& \le cr^{n-p} \Phi(r) \left\{1+ d^{1-p}\frac{r^{p}}{\Phi(r)}\frac{\Phi(R)}{R^{p}}[\tail(u_{-};R)]^{p-1}\right\},
\end{aligned}
\end{equation}
where $c = c(n,p,s,\tilde{s},\Lambda,L)>0$.
\end{lemma}

\begin{proof}
Let $\eta \in C^{\infty}_{0}(B_{3r/2})$ be a cut-off function such that $0 \le \eta \le 1$, $\eta \equiv 1$ in $B_{r}$ and $|D\eta| \le 4/r$. Testing \eqref{mainPDE} with $(u+d)^{1-p}\eta^{p}$, we have
\begin{equation*}
\begin{aligned}
0 & \le \int_{B_{2r}}\int_{B_{2r}}|u(x)-u(y)|^{p-2}(u(x)-u(y))\left[\frac{\eta^{p}(x)}{(u(x)+d)^{p-1}}-\frac{\eta^{p}(y)}{(u(y)+d)^{p-1}}\right]K(x,y)\,dx\,dy \\
& \quad\; + 2\int_{\mathbb{R}^{n}\setminus B_{2r}}\int_{B_{2r}}|u(x)-u(y)|^{p-2}(u(x)-u(y))\frac{\eta^{p}(x)}{(u(x)+d)^{p-1}}K(x,y)\,dx\,dy \\
& \eqqcolon I_{1} + I_{2}.
\end{aligned}
\end{equation*}
Following the proof of \cite[Lemma~1.3]{DKP16} (see also \cite[Proposition~3.10]{BDOR}), we have
\begin{equation*}
\begin{aligned}
I_{1} & \le -c^{-1}\int_{B_{2r}}\int_{B_{2r}}|\log(u(x)+d)-\log(u(y)+d)|^{p}\eta^{p}(y)K(x,y)\,dx\,dy \\
& \quad\; + c\int_{B_{2r}}\int_{B_{2r}}|\eta(x)-\eta(y)|^{p}K(x,y)\,dx\,dy \\
& \le -c^{-1}\int_{B_{r}}\int_{B_{r}}|\log(u(x)+d)-\log(u(y)+d)|^{p}K(x,y)\,dx\,dy + cr^{n-p}\Phi(r)
\end{aligned}
\end{equation*}
and
\begin{equation*}
\begin{aligned}
I_{2} & \le c\int_{\mathbb{R}^{n}\setminus B_{2r}}\int_{B_{2r}}\eta^{p}(x)K(x,y)\,dx\,dy + cd^{1-p}\int_{\mathbb{R}^{n}\setminus B_{R}}\int_{B_{2r}}u_{-}^{p-1}(y)\eta^{p}(x)K(x,y)\,dx\,dy \\
& \eqqcolon I_{2,1} + I_{2,2},
\end{aligned}
\end{equation*}
where $c= c(p)>1$.
To estimate $I_{2,1}$, we observe that $|y-x| \ge |y-x_{0}|-|x-x_{0}| \ge r/2$ for $x \in B_{3r/2}$ and $y \in \mathbb{R}^{n}\setminus B_{2r}$. 
Using this inequality, along with \eqref{out.int} and \eqref{phi.Phi}, we estimate
\begin{equation*}
I_{2,1} \le cr^{n}\sup_{x\in B_{3r/2}}\int_{\mathbb{R}^{n}\setminus B_{r/2}(x)}\frac{\phi(|x-y|)}{|x-y|^{n+p}}\,dy \le cr^{n-p}\phi(r) \le cr^{n-p}\Phi(r)
\end{equation*}
for a constant $c= c(n,p,s,\tilde{s},\Lambda,L)>0$. 
Moreover, since $\phi(t)/t^{n+p}$ is almost decreasing for $t\in(0,\infty)$ with constant $L$ and that \[ \frac{|y-x_{0}|}{|y-x|} \le 1+ \frac{|x-x_{0}|}{|y-x|} \le 1+ \frac{3r/2}{R-3r/2} \le 4 \]
for $x \in B_{3r/2}$ and $y \in \mathbb{R}^{n}\setminus B_{R}$, we obtain
\begin{equation*}
I_{2,2} \le cd^{1-p}r^{n}\int_{\mathbb{R}^{n}\setminus B_{R}}\frac{u_{-}^{p-1}(y)}{|y-x_{0}|^{p}}\frac{\phi(|y-x_{0}|)}{|y-x_{0}|^{n}}\,dy = cd^{1-p}r^{n}\frac{\Phi(R)}{R^{p}}[\tail(u_{-};R)]^{p-1}
\end{equation*}
for some $c= c(n,p,\Lambda,L)>0$. Combining all the estimates, we obtain \eqref{log.est}.
\end{proof}

As a consequence of Lemma~\ref{lem.log}, Theorem~\ref{thm:SoboPoin} and Corollary~\ref{cor.SP}, we have the following:
\begin{corollary}\label{cor.log}
Let $u \in \mathbb{W}^{\phi,p}(\Omega)$ be a weak supersolution to \eqref{mainPDE} under assumptions \eqref{kernel}, \eqref{Dini}, \eqref{adec} and \eqref{ainc}, which is nonnegative in a ball $B_{R}\equiv B_{R}(x_{0}) \subset \Omega$, and define
\[ v \coloneqq \min\{(\log(a+d)-\log(u+d))_{+},\log b\} \]
for constants $a,d>0$ and $b>1$. Then for any $r \in (0,R/2]$,
\begin{equation*}
\mean{B_{r}}|v-(v)_{B_{r}}|^{p}\,dx \le c + cd^{1-p}\frac{r^{p}}{\Phi(r)}\frac{\Phi(R)}{R^{p}}[\tail(u_{-};R)]^{p-1}
\end{equation*}
holds for a constant $c= c(n,p,s,\tilde{s},\Lambda,L)>0$.
\end{corollary}

\subsection{Local boundedness}
Here we prove Theorem~\ref{thm.bdd}.

\begin{proof}[Proof of Theorem~\ref{thm.bdd}]

Let $B_{r}\equiv B_{r}(x_{0})\subset \Omega$ be a fixed ball, and let $\varepsilon \in (0,1]$. 
For $r/2 \le \rho < \sigma \le r$ and $k>0$, we set
\[ A^{+}(k,\rho) \coloneqq \{x\in B_{\rho}:u(x) \ge k \}. \]
Observe that for any $0<h<k$ and $x \in A^{+}(k,\rho)\subset A^{+}(h,\rho)$, 
\[ (u(x)-h)_{+} = u(x)-h \ge k-h \]
and
\[ (u(x)-h)_{+} = u(x)-h \ge u(x)-k = (u(x)-k)_{+}. \]
Thus, we have
\begin{equation}\label{aplus.est}
|A^{+}(k,\rho)| \le \int_{A^{+}(k,\rho)}\frac{(u-h)_{+}^{p}}{(k-h)^{p}}\,dx \le \frac{1}{(k-h)^{p}}\int_{A^{+}(h,\sigma)}(u-h)_{+}^{p}\,dx
\end{equation}
and
\begin{equation}\label{uk.est}
\int_{B_{\sigma}}(u-k)_{+}\,dx \le \int_{B_{\sigma}}(u-h)_{+}\frac{(u-h)_{+}^{p-1}}{(k-h)^{p-1}}\,dx = \frac{1}{(k-h)^{p-1}}\int_{B_{\sigma}}(u-h)_{+}^{p}\,dx.
\end{equation}

In view of Remark~\ref{rmk:adec}, we may assume that $s p< n$ without loss of generality, and set $\kappa \coloneqq n/(n-s p) >1$. By H\"older's inequality and Theorem~\ref{thm:SoboPoin} with $v=(u-k)_+$, 
\begin{equation*}
\begin{aligned}
& \mean{B_{\rho}}(u-k)_{+}^{p}\,dx \le \left(\frac{|A^{+}(k,\rho)|}{|B_{\rho}|}\right)^{(\kappa-1)/\kappa}\left(\mean{B_{\rho}}(u-k)_{+}^{p\kappa}\,dx\right)^{1/\kappa} \\
& \le c\left(\frac{|A^{+}(k,\rho)|}{|B_{\rho}|}\right)^{(\kappa-1)/\kappa}\frac{r^{p}}{\Phi(r)}\mean{B_{\rho}}\int_{B_{\rho}}\frac{|(u(x)-k)_{+}-(u(y)-k)_{+}|^{p}}{|x-y|^{p}}\frac{\phi(|x-y|)}{|x-y|^{n}}\,dxdy \\
& \quad + c\left(\frac{|A^{+}(k,\rho)|}{|B_{\rho}|}\right)^{(\kappa-1)/\kappa}\mean{B_{\rho}}(u-k)_{+}^{p}\,dx.
\end{aligned}
\end{equation*}
Using this and Lemma~\ref{lem.ccp}, we obtain
\begin{equation*}
\begin{aligned}
& \frac{r^{p}}{\Phi(r)}\mean{B_{\rho}}\int_{B_{\rho}}\frac{|(u(x)-k)_{+}-(u(y)-k)_{+}|^{p}}{|x-y|^{p}}\frac{\phi(|x-y|)}{|x-y|^{n}}\,dx\,dy \\
& \le c\left(\frac{\sigma}{\sigma-\rho}\right)^{p}\mean{B_{\sigma}}(u-k)_{+}^{p}\,dx + c\left(\frac{\sigma}{\sigma-\rho}\right)^{n+\tilde{s}p}[\tail((u-k)_{+};\sigma)]^{p-1}\mean{B_{\sigma}}(u-k)_{+}\,dx.
\end{aligned}
\end{equation*}
Combining the last two displays, and then applying \eqref{aplus.est} and \eqref{uk.est}, we arrive at
\begin{equation}\label{cacciokh}
\begin{aligned}
\mean{B_{\rho}}(u-k)_{+}^{p}\,dx & \le \frac{c}{(k-h)^{p(\kappa-1)/\kappa}}\left[\left(\frac{\sigma}{\sigma-\rho}\right)^{p}+\left(\frac{\sigma}{\sigma-\rho}\right)^{n+\tilde{s}p}\frac{[\tail(u_{+};r/2)]^{p-1}}{(k-h)^{p-1}} + 1\right] \\
& \quad\; \cdot \left(\mean{B_{\sigma}}(u-h)_{+}^{p}\,dx\right)^{1+(\kappa-1)/\kappa}.
\end{aligned}
\end{equation}
Now, for $i \in \mathbb{N}\cup\{0\}$ and $k_{0}>0$ with 
\begin{equation}\label{choosek01}
2k_{0} \ge \varepsilon\tail(u_{+};r/2),
\end{equation}
we set
\[ \sigma_{i} \coloneqq \frac{r}{2}(1+2^{-i}), \qquad k_{i} \coloneqq 2k_{0}(1-2^{-i-1}), \qquad y_{i} \coloneqq \frac{1}{|B_{r}|}\int_{A^{+}(k_{i},\sigma_{i})}\left[\frac{(u-k_{i})_{+}}{k_{0}}\right]^{p}\,dx. \]
Choosing $k = k_{i+1}$, $h = k_{i}$, $\rho = \sigma_{i+1}$ and $\sigma = \sigma_{i}$ in \eqref{cacciokh}, and then dividing both sides of the resulting inequality by $k_{0}^{p}$, we obtain
\begin{equation*}
y_{i+1} \le \tilde{c}\varepsilon^{1-p}2^{i[p(\kappa-1)/\kappa +n+\tilde{s}p +p-1]}y_{i}^{1+(\kappa-1)/\kappa}
\end{equation*}
for a constant $\tilde{c} = \tilde{c}(n,s,\tilde{s},p,\Lambda,L)>0$. According to Lemma~\ref{lemseq}, if
\begin{equation}\label{choosek02}
y_{0} = \frac{1}{k_{0}^{p}}\mean{B_{r}}(u-k_{0})_{+}^{p}\,dx \le (\tilde{c}\varepsilon^{1-p})^{-\kappa/(\kappa-1)}2^{-[p(\kappa-1)/\kappa+n+\tilde{s}p + p-1]\kappa^{2}/(\kappa-1)^{2}},
\end{equation}
then $\lim_{i\rightarrow\infty}y_{j}=0$ and therefore $u \le 2k_{0}$ a.e. in $B_{r/2}$. Now, for a sufficiently large constant $c_{b} = c_{b}(n,s,\tilde{s},p,\Lambda,L)>0$,
\[ 2k_{0} = c_{b}\varepsilon^{-\frac{(p-1)\kappa}{p(\kappa-1)}}\left(\mean{B_{r}}u_{+}^{p}\,dx\right)^{1/p} + \varepsilon\tail(u_{+};r/2) \]
satisfies both \eqref{choosek01} and \eqref{choosek02}; this leads to \eqref{sup.est} and \eqref{sup.est2}.
\end{proof}

\section{Expansion of positivity}\label{sec5}
We start this section with an expansion of positivity lemma which will be used in the proof of Theorems \ref{thm.hol} and \ref{thm.harnack}.

\begin{lemma}\label{lem.positivity}
Let $u \in \mathbb{W}^{\phi,p}(\Omega)$ be a weak supersolution to \eqref{mainPDE} under assumptions \eqref{kernel}, \eqref{Dini}, \eqref{adec} and \eqref{ainc}, which is nonnegative in a ball $B_{R}\equiv B_{R}(x_{0}) \subset \Omega$. Let $k>0$ and $r\in(0,R/4]$, and assume that 
\begin{equation}\label{density.assumption}
|B_{2r}\cap \{u \ge k\}| \ge \sigma|B_{2r}|
\end{equation}
for some $\sigma\in(0,1)$. Then there exists a constant $\delta = \delta(n,p,s,\tilde{s},\Lambda,L,\sigma) \in (0,1/4)$ such that the following holds: if
\begin{equation}\label{tail.small}
\left(\frac{r^{p}}{\Phi(r)}\frac{\Phi(R)}{R^{p}}\right)^{1/(p-1)}\tail(u_{-};R) \le \delta k, 
\end{equation}
then
\begin{equation}\label{lower.bound}
\inf_{B_{r}}u \ge \delta k.
\end{equation}
\end{lemma}
\begin{proof}
\textit{Step 1: A density estimate. } 
We first show that there exists a constant $\bar{c} = \bar{c}(n,p,s,\tilde{s},\Lambda,L)>0$ such that
\begin{equation}\label{density.est}
\left|B_{2r}\cap\left\{u \le 2\delta k\right\}\right| \le \frac{\bar{c}}{\sigma \log(1/3\delta)}|B_{2r}|
\end{equation}
holds for any $\delta \in (0,1/4)$. We set
\[ d\coloneqq \frac{1}{2}\left(\frac{r^{p}}{\Phi(r)}\frac{\Phi(R)}{R^{p}}\right)^{1/(p-1)}\tail(u_{-};R) \quad \text{and then} \quad  v\coloneqq \left[\min\left\{\log\frac{1}{3\delta},\log\frac{k+d}{u+d}\right\}\right]_{+}. \]
With this choice of $v$, Corollary~\ref{cor.log} and H\"older's inequality imply that
\begin{equation} \label{loglog}
\mean{B_{2r}}|v-(v)_{B_{2r}}|\,dx \le \left(\mean{B_{2r}}|v-(v)_{B_{2r}}|^{p}\,dx\right)^{1/p} \le \bar{c}
\end{equation}
holds for a constant $\bar{c}= \bar{c}(n,p,s,\tilde{s},\Lambda,L)>0$.
Moreover, by \eqref{density.assumption} and the definition of $v$,
\begin{equation*}
|B_{2r}\cap\{v=0\}| \ge \sigma|B_{2r}|
\end{equation*}
and so
\begin{equation*}
\log\frac{1}{3\delta} = \frac{1}{|B_{2r}\cap\{v=0\}|}\int_{B_{2r}\cap\{v=0\}}\left(\log\frac{1}{3\delta}-v\right)\,dx 
 \le \frac{1}{\sigma}\left(\log\frac{1}{3\delta}-(v)_{B_{2r}}\right).
\end{equation*}
Integrating this inequality over the set $B_{2r}\cap\{v=\log(1/3\delta)\}$ and then applying \eqref{loglog}, we obtain
\begin{equation}\label{asdf}
\log\frac{1}{3\delta}\left|B_{2r}\cap\left\{v=\log\frac{1}{3\delta}\right\}\right| \le \frac{1}{\sigma}\int_{B_{2r}\cap\{v=\log(1/3\delta)\}}|v-(v)_{B_{2r}}|\,dx \le \frac{\bar{c}}{\sigma}|B_{2r}|.
\end{equation}
By the definitions of $v$ and $d$ and  \eqref{tail.small}, 
\[ \{v=\log(1/3\delta)\} = \{u+d \le 3\delta(k+d)\} \supset \{u \le 2\delta k \}, \]
which along with \eqref{asdf} implies \eqref{density.est}.

\textit{Step 2: A pointwise bound. } We now prove \eqref{lower.bound} by choosing $\delta \in (0,1/4)$ in a suitable way.  
For each $j \in \mathbb{N}\cup\{0\}$, we set
\[ \rho_{j}\coloneqq (1+2^{-j})r, \qquad B_{j}\coloneqq B_{\rho_{j}}(x_{0}) \]
and
\[ \ell_{j}\coloneqq (1+2^{-j})\delta k, \qquad w_{j} \coloneqq (\ell_{j}-u)_{+}, \qquad A_{j}\coloneqq \frac{|B_{j}\cap\{u <\ell_{j}\}|}{|B_{j}|}. \]
In this setting, \eqref{density.est} reads as
\begin{equation}\label{A0.small} 
A_{0} = \frac{|B_{0}\cap\{u<2\delta k\}|}{|B_{0}|} \le \frac{\bar{c}}{\sigma\log(1/3\delta)}. 
\end{equation}
We also observe that
\[w_{j} = (\ell_{j}-u)_{+} \ge (\ell_{j}-\ell_{j+1})\chi_{\{u<\ell_{j+1}\}} \ge 2^{-j-2}\ell_{j}\chi_{\{u<\ell_{j+1}\}}. \] 
Without loss of generality, we may assume that $s p<n$ (see Remark~\ref{rmk:adec}). 
Thus, setting $\kappa \coloneqq n/(n-s p) >1$ , Theorem~\ref{thm:SoboPoin} implies
\begin{equation}\label{Aj.1st}
\begin{aligned}
2^{-(j+2)p}\ell_{j}^{p}A_{j+1}^{1/\kappa} & \le \left(\mean{B_{j+1}}w_{j}^{p\kappa}\,dx\right)^{1/\kappa} \\
& \le c\frac{\rho_{j+1}^{p}}{\Phi(\rho_{j+1})}\mean{B_{j+1}}\int_{B_{j+1}}\frac{|w_{j}(x)-w_{j}(y)|^{p}}{|x-y|^{p}}\frac{\phi(|x-y|)}{|x-y|^{n}}\,dx\,dy + c\mean{B_{j+1}}w_{j}^{p}\,dx \\
& \eqqcolon I_{1} + I_{2}.
\end{aligned}
\end{equation}
It is straightforward to estimate $I_{2}$ as
\[ I_{2} = c\mean{B_{j+1}}w_{j}^{p}\,dx \le c\mean{B_{j}}w_{j}^{p}\,dx \le \frac{c}{|B_{j}|}\int_{B_{j}\cap\{u\le \ell_{j}\}}\ell_{j}^{p}\,dx \le c\ell_{j}^{p}A_{j}, \]
and a similar calculation also gives
\begin{equation}\label{wj.est} 
\mean{B_{j}}w_{j}\,dx \le c\ell_{j}A_{j}.
\end{equation}
As for $I_{1}$, Lemma~\ref{lem.ccp} together with the fact that $\Phi(\rho_i)\le 2^pL\Phi(\rho_{i+1})$ implies 
\begin{equation*}
\begin{aligned}
I_{1} & \le c\left(\frac{\rho_{j}}{\rho_{j}-\rho_{j+1}}\right)^{p}\mean{B_{j}}w_{j}^{p}\,dx \\
& \quad + c\left(\frac{\rho_{j}}{\rho_{j}-\rho_{j+1}}\right)^{n+\tilde{s}p}\mean{B_{j}}w_{j}\,dx\cdot\frac{\rho_{j}^{p}}{\Phi(\rho_{j})}\int_{\mathbb{R}^{n}\setminus B_{j}}\frac{(\ell_{j}+u_{-}(y))^{p-1}}{|y-x_{0}|^{p}}\frac{\phi(|y-x_{0}|)}{|y-x_{0}|^{n}}\,dy,
\end{aligned}
\end{equation*}
where we have also used the fact that $w_{j} = (\ell_{j}-u)_{+} \le \ell_{j} + u_{-}$. Then, using \eqref{out.int}, \eqref{tail.small} and the fact that $\delta k \le \ell_{j}$, we get
\begin{equation*}
\begin{aligned}
& \frac{\rho_{j}^{p}}{\Phi(\rho_{j})}\int_{\mathbb{R}^{n}\setminus B_{j}}\frac{(\ell_{j}+u_{-}(y))^{p-1}}{|y-x_{0}|^{p}}\frac{\phi(|y-x_{0}|)}{|y-x_{0}|^{n}}\,dy \\
& \le c\frac{\rho_{j}^{p}}{\Phi(\rho_{j})}\left(\int_{\mathbb{R}^{n}\setminus B_{j}}\frac{\ell_{j}^{p-1}}{|y-x_{0}|^{p}}\frac{\phi(|y-x_{0}|)}{|y-x_{0}|^{n}}\,dy + \int_{\mathbb{R}^{n}\setminus B_{R}}\frac{u_{-}^{p-1}(y)}{|y-x_{0}|^{p}}\frac{\phi(|y-x_{0}|)}{|y-x_{0}|^{n}}\,dy\right) \\
& \le c\ell_{j}^{p-1} + c\frac{r^{p}}{\Phi(r)}\frac{\Phi(R)}{R^{p}}[\tail(u_{-};R)]^{p-1} \\
& \le c\ell_{j}^{p-1},
\end{aligned}
\end{equation*}
which along with \eqref{wj.est} gives
\begin{equation*}
I_{1} \le c2^{jp}\ell_{j}^{p}A_{j} + c2^{j(n+\tilde{s}p)}\ell_{j}^{p}A_{j}.
\end{equation*}
Connecting the estimates found for $I_{1}$ and $I_{2}$ to \eqref{Aj.1st}, we arrive at
\[ A_{j+1} \le c_{*}2^{j(n+\tilde{s}p)\kappa}A_{j}^{\kappa} \]
for some $c_{*} = c_{*}(n,p,s,\tilde{s},\Lambda,L)>0$. Now, recalling \eqref{A0.small}, we choose $\delta = \delta(n,p,s,\tilde{s},\Lambda,L)$ as
\[ \delta =\frac{1}{4}\exp\left\{-\frac{\bar{c}}{\sigma}c_{*}^{1/(\kappa-1)}2^{(n+\tilde{s}p)\kappa/(\kappa-1)^{2}}\right\} \in \left(0,\frac{1}{4}\right) \]
so that
\begin{equation*}
A_{0} = \frac{|B_{2r}\cap\{u<\ell_{0}\}|}{|B_{2r}|} \le \frac{\bar{c}}{\sigma}\frac{1}{\log(1/3\delta)}
 \le c_{*}^{-1/(\kappa-1)}2^{-(n+\tilde{s}p)\kappa/(\kappa-1)^{2}}.
\end{equation*}
Consequently, Lemma~\ref{lemseq} implies $\lim_{j\to\infty}A_{j} = 0$ and so $u \ge \delta k$ a.e. in $B_{r}$. 
\end{proof}

\subsection{H\"older regularity}
We prove Theorem~\ref{thm.hol}. 
\begin{proof}[Proof of Theorem~\ref{thm.hol}]
Let $B_{r}\equiv B_{r}(x_{0})\subset \Omega$ be a fixed ball, and set
\[ k_{0} \coloneqq 2\left[c_{b}\left(\mean{B_{r}}|u|^{p}\,dx\right)^{1/p} + \tail(u;r/2)\right], \]
where $c_{b} = c_{b}(n,p,s,\tilde{s},\Lambda,L)>0$ is the constant determined in Theorem~\ref{thm.bdd}.
In order to obtain \eqref{holder.est}, it suffices to show that there exist constants $\alpha \in (0,1)$ and $\tau \in (0,1/4)$, both depending only on $n$, $p$, $s$, $\tilde{s}$, $\Lambda$ and $L$, such that
\begin{equation}\label{osc.est}
\osc_{B_{\tau^{j}r/2}}u \le \tau^{\alpha j}k_{0}
\end{equation}
holds for every $j \in \mathbb{N}\cup\{0\}$. 

We proceed by strong induction on $j$. First, the estimate \eqref{sup.est2} with $\delta=1$ implies that \eqref{osc.est} holds for $j=0$. We now assume that \eqref{osc.est} holds for every $j\in\{0,1,\ldots,i\}$, with $i \in \mathbb{N}\cup\{0\}$ being fixed, and then show that it holds for $j=i+1$ as well. 

For each $j \in \mathbb{N}\cup\{0\}$, we set
\[ r_{j} \coloneqq \tau^{j}\frac{r}{2}, \qquad B_{j} \coloneqq B_{r_{j}}(x_{0}), \qquad k_{j} \coloneqq \left(\frac{r_{j}}{r_{0}}\right)^{\alpha}k_{0} = \tau^{\alpha j}k_{0}, \]
with the values of $\alpha \in (0,1)$ and $\tau \in (0,1/4)$ to be determined later, and then
\[ M_{j} \coloneqq \sup_{B_{j}}u, \qquad m_{j} \coloneqq \inf_{B_{j}}u. \]
Observe that either
\begin{equation}\label{alt1}
\left|2B_{i+1}\cap\left\{u-m_{i} \ge \tfrac{1}{2}k_{i}\right\}\right| \ge \tfrac{1}{2}\left|2B_{i+1}\right|
\end{equation}
or
\begin{equation}\label{alt2}
\left|2B_{i+1}\cap\left\{k_{i} - (u-m_{i}) \ge \tfrac{1}{2}k_{i}\right\}\right| \ge \tfrac{1}{2}\left|2B_{i+1}\right|
\end{equation}
must hold; we accordingly define
\begin{equation*}
u_{i} \coloneqq 
\begin{cases}
u - m_{i} & \text{ if \eqref{alt1} holds,} \\
k_{i} - (u-m_{i}) & \text{ if \eqref{alt2} holds}.
\end{cases}
\end{equation*}
Then $u_{i}$ is a weak solution to \eqref{mainPDE} that satisfies $u_{i} \ge 0$ in $B_{i}$ and
\begin{equation}\label{levelset.i+1}
\left|2B_{i+1}\cap\left\{u_{i} \ge \tfrac{1}{2}k_{i}\right\}\right| \ge \tfrac{1}{2}\left|2B_{i+1}\right|.
\end{equation}
Moreover, it holds that
\[ |u_{i}| \le M_{j} - m_{j} + k_{i} \le k_{j} + k_{i} \le 2k_{j} \;\; \text{a.e. in}\;\; B_{j} \quad \text{for any}\;\; j \in \{0,1,\ldots,i\} \]
and that 
\[ |u_{i}| \le |u| + 2k_{0} \;\; \text{a.e. in} \;\; \mathbb{R}^{n}\setminus B_{0}. \] 
Using these inequalities along with \eqref{out.int}, we estimate
\begin{equation*}
\begin{aligned}
& \frac{\Phi(r_{i})}{r_{i}^{p}}[\tail(u_{i};r_{i})]^{p-1} \\
& = \sum_{j=1}^{i}\int_{B_{j-1}\setminus B_{j}}\frac{|u_{i}(y)|^{p-1}}{|y-x_{0}|^{p}}\frac{\phi(|y-x_{0}|)}{|y-x_{0}|^{n}}\,dy + \int_{\mathbb{R}^{n}\setminus B_{0}}\frac{|u_{i}(y)|^{p-1}}{|y-x_{0}|^{p}}\frac{\phi(|y-x_{0}|)}{|y-x_{0}|^{n}}\,dy \\
& \le \sum_{j=1}^{i}\int_{B_{j-1}\setminus B_{j}}\frac{(2k_{j-1})^{p-1}}{|y-x_{0}|^{p}}\frac{\phi(|y-x_{0}|)}{|y-x_{0}|^{n}}\,dy + \int_{\mathbb{R}^{n}\setminus B_{0}}\frac{(|u(y)|+2k_{0})^{p-1}}{|y-x_{0}|^{p}}\frac{\phi(|y-x_{0}|)}{|y-x_{0}|^{n}}\,dy \\
& \le c\sum_{j=1}^{i}(2k_{j-1})^{p-1}\frac{\Phi(r_{j})}{r_{j}^{p}} + c\frac{\Phi(r_{0})}{r_{0}^{p}}[\tail(u;r_{0})]^{p-1} + ck_{0}^{p-1}\frac{\Phi(r_{0})}{r_{0}^{p}} \\
& \le c\sum_{j=1}^{i}k_{j-1}^{p-1}\frac{\Phi(r_{j})}{r_{j}^{p}}
\end{aligned}
\end{equation*}
for a constant $c = c(n,p,s,\tilde{s},L)>0$. Thus, if $\alpha$ is so small that
\begin{equation}\label{gamma.choice1}
\alpha \le \frac{s p}{2(p-1)},
\end{equation}
then, using \eqref{adec}, 
\begin{equation*}
\begin{aligned}
\frac{r_{i+1}^{p}}{\Phi(r_{i+1})}\frac{\Phi(r_{i})}{r_{i}^{p}}[\tail((u_{i})_{-};r_{i})]^{p-1} & \le L \tau^{s p}[\tail(u_{i};r_{i})]^{p-1} \\
& \le c\tau^{s p}\frac{r_{i}^{p}}{\Phi(r_{i})}\sum_{j=1}^{i}k_{j-1}^{p-1}\frac{\Phi(r_{j})}{r_{j}^{p}} \\
& \le ck_{i}^{p-1}\sum_{j=1}^{i}\tau^{(1+i-j)[s p-\alpha(p-1)]} \\
& \le ck_{i}^{p-1}\sum_{j=1}^{i}\tau^{js p/2} \le c\frac{\tau^{s p/2}}{1-\tau^{s p/2}}k_{i}^{p-1},
\end{aligned}
\end{equation*}
holds for a constant $c= c(n,p,s,\tilde{s},L)$. We now choose $\tau = \tau(n,p,s,\tilde{s},\Lambda,L) \in (0,1/4)$ so small that
\[ \left(\frac{r_{i+1}^{p}}{\Phi(r_{i+1})}\frac{\Phi(r_{i})}{r_{i}^{p}}\right)^{1/(p-1)}\tail((u_{i})_{-};r_{i}) \le \left(c\frac{\tau^{s p/2}}{1-\tau^{s p/2}}\right)^{1/(p-1)}k_{i} \le \frac{\delta}{2}k_{i}, \]
where $\delta = \delta(n,p,s,\tilde{s},\Lambda,L) \in (0,1/4)$ is the constant determined in Lemma~\ref{lem.positivity} with the choice $\sigma = 1/2$. 
Recalling \eqref{levelset.i+1}, we use Lemma~\ref{lem.positivity} with the choices $\sigma = 1/2$, $k = k_{i}/2$, $B_{R} = B_{i}$ and $B_{r} = B_{i+1}$ in order to get 
\[ \inf_{B_{i+1}}u_{i+1} \ge \delta k_{i}/2.\]
This in turn implies the following:
\begin{itemize}
\item[(i)] If \eqref{alt1} holds, then $m_{i+1} - m_{i} \ge \delta k_{i}/2$ and so 
\begin{equation*}
M_{i+1} - m_{i+1} \le M_{i} - m_{i} - (m_{i+1}-m_{i}) = \osc_{B_{i}}u - (m_{i+1}-m_{i}) \le \left(1-\frac{\delta}{2}\right)k_{i}.
\end{equation*}
\item[(ii)] If \eqref{alt2} holds, then $k_{i} - M_{i+1} + m_{i} \ge \delta k_{i}/2$ and so
\begin{equation*}
M_{i+1}-m_{i+1} \le M_{i+1}-m_{i} \le \left(1-\frac{\delta}{2}\right)k_{i}.
\end{equation*}
\end{itemize}
Namely, in any case we obtain 
\[ \osc_{B_{i+1}}u \le \left(1-\frac{\delta}{2}\right)\tau^{-\alpha}k_{i+1}. \]
Then we finally choose $\alpha = \alpha(n,p,s,\tilde{s},\Lambda,L)\in (0,1)$ small enough to satisfy \eqref{gamma.choice1} and $1-\delta/2\le \tau^{\alpha}$. This implies \eqref{osc.est} for $j=i+1$, and the proof of Theorem~\ref{thm.hol} is complete.
\end{proof}

\subsection{Nonlocal Harnack inequality}
We start by introducing a nonlocal version of weak Harnack inequality with a (small) fixed exponent $p_0>0$. Note that this can be obtained  using almost the same method as in \cite{BDOR,Coz17,DKP14}, with considering the properties of the function $\phi$ that we have used. Therefore, we shall omit the proof of the next lemma.  
\begin{lemma}\label{lem.wh}
Let $u \in \mathbb{W}^{\phi,p}(\Omega)$ be a weak supersolution to \eqref{mainPDE} under assumptions \eqref{kernel}, \eqref{Dini}, \eqref{adec} and \eqref{ainc}, which is nonnegative in a ball $B_{R}\equiv B_{R}(x_{0})\subset \Omega$. Then there exist constants $p_{0} \in (0,1)$ and $c\ge 1$, both depending only on $n$, $p$, $s$, $\tilde{s}$, $\Lambda$ and $L$, such that
\begin{equation}\label{wh.est}
\left(\mean{B_{r}}u^{p_{0}}\,dx\right)^{1/p_{0}} \le c\inf_{B_{r}}u + c\left(\frac{r^{p}}{\Phi(r)}\frac{\Phi(R)}{R^{p}}\right)^{1/(p-1)}\tail(u_{-};R)
\end{equation}
holds whenever $r \in (0,R]$.
\end{lemma}

We next obtain a tail estimate.
\begin{lemma}\label{lem.tail}
Let $u \in \mathbb{W}^{\phi,p}(\Omega)$ be a weak solution to \eqref{mainPDE} under assumptions \eqref{kernel}, \eqref{Dini}, \eqref{adec} and \eqref{ainc}, which is nonnegative in a ball $B_{R} \equiv B_{R}(x_{0}) \subset \Omega$. Then 
\begin{equation}\label{tail.est}
\tail(u_{+};r) \le c\sup_{B_{r}}u + c\left(\frac{r^{p}}{\Phi(r)}\frac{\Phi(R)}{R^{p}}\right)^{1/(p-1)}\tail(u_{-};R)
\end{equation}
holds whenever $r \in (0,R]$, where $c = c(n,p,s,\tilde{s},\Lambda,L)>0$.
\end{lemma}
\begin{proof}
We take a cut-off function $\eta\in C^{\infty}_{0}(B_{r})$ such that $0 \le \eta \le 1$, $\eta\equiv1$ in $B_{r/2}$ and $|D\eta| \le 8/r$. Testing \eqref{mainPDE} with $\zeta \equiv (u-2k)\eta^{p}$, where $k \coloneqq \sup_{B_{r}}u$, we have
\begin{equation*}
\begin{aligned}
0 & = \int_{B_{r}}\int_{B_{r}}|u(x)-u(y)|^{p-2}(u(x)-u(y))(\zeta(x)-\zeta(y))K(x,y)\,dx\,dy \\
& \quad\; + 2\int_{\mathbb{R}^{n}\setminus B_{r}}\int_{B_{r}}|u(x)-u(y)|^{p-2}(u(x)-u(y))(u(x)-2k)\eta^{p}(x)K(x,y)\,dx\,dy \\
& \eqqcolon I_{1} + I_{2}.
\end{aligned}
\end{equation*}
We observe that $I_{1}$ can be estimated by similar calculations as in the proof of \cite[Lemma~4.2]{DKP14}:
\begin{equation*}
\begin{aligned}
I_{1}  \ge -ck^{p}\int_{B_{r}}\int_{B_{r}}\frac{|\eta(x)-\eta(y)|^{p}}{|x-y|^{p}}\frac{\phi(|x-y|)}{|x-y|^{n}}\,dx\,dy 
& \ge -ck^{p}r^{n-p}\int_{B_{2r}(y)}\frac{\phi(|x-y|)}{|x-y|^{n}}\,dx \\
& \ge -ck^{p}r^{n-p} \Phi(r).
\end{aligned}
\end{equation*}
We next split $I_{2}$ as
\begin{equation*}
\begin{aligned}
I_{2} & \ge \int_{\mathbb{R}^{n}\setminus B_{r}}\int_{B_{r}}k(u(y)-k)_{+}^{p-1}\eta^{p}(x)K(x,y)\,dx\,dy \\
& \quad\; - \int_{\mathbb{R}^{n}\setminus B_{r}}\int_{B_{r}}2k\chi_{\{u(y)<k\}}(u(x)-u(y))_{+}^{p-1}\eta^{p}(x)K(x,y)\,dx\,dy \\
& \eqqcolon I_{2,1} - I_{2,2}.
\end{aligned}
\end{equation*}
In order to estimate $I_{2,1}$, we observe that $|x-y| \ge |y-x_{0}| - |x-x_{0}| \ge \frac{1}{2}|y-x_{0}| \ge \frac{1}{2}r$ and $|x-y| \le |x-x_0| + |y-x_0| \le \frac{3}{2}|y-x_0|$ for any $x\in B_{r/2}$ and $y \in \mathbb{R}^{n}\setminus B_{r}$. Using this fact together with \eqref{adec}, \eqref{ainc} and \eqref{out.int}, we have
\begin{equation*}
\begin{aligned}
I_{2,1} & \ge c^{-1}k\int_{\mathbb{R}^{n}\setminus B_{r}}\int_{B_{r}}\frac{u_{+}^{p-1}(y)}{|x-y|^{p}}\eta^{p}(x)\frac{\phi(|x-y|)}{|x-y|^{n}}\,dxdy - ck^{p}\int_{\mathbb{R}^{n}\setminus B_{r}}\int_{B_{r}}\frac{\eta^{p}(x)}{|x-y|^{p}}\frac{\phi(|x-y|)}{|x-y|^{n}}\,dx\,dy \\
& \ge c^{-1}kr^{n}\int_{\mathbb{R}^{n}\setminus B_{r}}\frac{u_{+}^{p-1}(y)}{|y-x_{0}|^{p}}\frac{\phi(|y-x_{0}|)}{|y-x_{0}|^{n}}\,dy - ck^{p}r^{n}\frac{\Phi(r)}{r^{p}} \\
& = c^{-1}kr^{n}\frac{\Phi(r)}{r^{p}}[\tail(u_{+};r)]^{p-1} - ck^{p}r^{n}\frac{\Phi(r)}{r^{p}}.
\end{aligned}
\end{equation*} 
In a similar way, we estimate $I_{2,2}$ as
\begin{equation*}
\begin{aligned}
I_{2,2} & \le ck\int_{B_{R}\setminus B_{r}}\int_{B_{r}}\frac{k^{p-1}}{|x-y|^{p}}\eta^{p}(x)\frac{\phi(|x-y|)}{|x-y|^{n}}\,dx\,dy \\
& \quad\; + ck\int_{\mathbb{R}^{n}\setminus B_{R}}\int_{B_{r}}\frac{(k+u_{-}(y))^{p-1}}{|x-y|^{p}}\eta^{p}(x)\frac{\phi(|x-y|)}{|x-y|^{n}}\,dx\,dy \\
& \le ck^{p}r^{n}\frac{\Phi(r)}{r^{p}} + ckr^{n}\frac{\Phi(R)}{R^{p}}[\tail(u_{-};R)]^{p-1}.
\end{aligned}
\end{equation*}
Combining the estimates found for $I_{1}$, $I_{2,1}$ and $I_{2,2}$, we obtain \eqref{tail.est}.
\end{proof}

We now prove Theorem~\ref{thm.harnack}.

\begin{proof}[Proof of Theorem~\ref{thm.harnack}]
Let $u$ be a weak solution to \eqref{mainPDE} which is nonnegative in a ball $B_{R}\equiv B_{R}(x_{0})\subset\Omega$, and $r\in (0, R/2]$. 
Note that we can also derive the following estimate as a variation of \eqref{sup.est}:
\begin{equation*}
\sup_{B_{\sigma_{1}r}}u \le \frac{c_{\varepsilon}}{(\sigma_{2}-\sigma_{1})^{n/p}}\left(\mean{B_{\sigma_{2}r}}u^{p}\,dx\right)^{1/p} + \varepsilon\tail(u_{+};\sigma_{1}r)
\end{equation*}
whenever $1 \le \sigma_{1} < \sigma_{2} \le 2$ and $\varepsilon \in (0,1)$, where $c= c(n,p,s,\tilde{s},\Lambda,L)$ and $c_{\varepsilon}= c_{\varepsilon}(n,p,s,\tilde{s},\Lambda,L,\varepsilon)$ are positive constants. We next apply Lemma~\ref{lem.tail} to the second term in the right-hand side, which gives
\begin{equation*}
\begin{aligned}
\sup_{B_{\sigma_{1}r}}u & \le \frac{c_{\varepsilon}}{(\sigma_{2}-\sigma_{1})^{n/p}}\left(\mean{B_{\sigma_{2}r}}u^{p}\,dx\right)^{1/p} + c\varepsilon\sup_{B_{\sigma_{1}r}}u + c\varepsilon\left(\frac{r^{p}}{\Phi(r)}\frac{\Phi(R)}{R^{p}}\right)^{1/(p-1)}\tail(u_{-};R).
\end{aligned}
\end{equation*}
Choosing $\varepsilon=1/(2c)$, reabsorbing terms and then using Young's inequality, we obtain
\begin{equation*}
\begin{aligned}
\sup_{B_{\sigma_{1}r}}u & \le \frac{c}{(\sigma_{2}-\sigma_{1})^{n/p}}\left(\mean{B_{\sigma_{2}r}}u^{p}\,dx\right)^{1/p} + c\left(\frac{r^{p}}{\Phi(r)}\frac{\Phi(R)}{R^{p}}\right)^{1/(p-1)}\tail(u_{-};R) \\
& \le \frac{c}{(\sigma_{2}-\sigma_{1})^{n/p}}\left(\sup_{B_{\sigma_{2}r}}u\right)^{(p-p_{0})/p}\left(\mean{B_{\sigma_{2}r}}u^{p_{0}}\,dx\right)^{1/p} + c\left(\frac{r^{p}}{\Phi(r)}\frac{\Phi(R)}{R^{p}}\right)^{1/(p-1)}\tail(u_{-};R) \\
& \le \frac{c}{(\sigma_{2}-\sigma_{1})^{n/p}}\left(\sup_{B_{\sigma_{2}r}}u\right)^{(p-p_{0})/p}\left(\mean{B_{2r}}u^{p_{0}}\,dx\right)^{1/p} + c\left(\frac{r^{p}}{\Phi(r)}\frac{\Phi(R)}{R^{p}}\right)^{1/(p-1)}\tail(u_{-};R) \\
& \le \frac{1}{2}\sup_{B_{\sigma_{2}r}}u + \frac{c}{(\sigma_{2}-\sigma_{1})^{n/p_{0}}}\left(\mean{B_{2r}}u^{p_{0}}\,dx\right)^{1/p_{0}} + c\left(\frac{r^{p}}{\Phi(r)}\frac{\Phi(R)}{R^{p}}\right)^{1/(p-1)}\tail(u_{-};R)
\end{aligned}
\end{equation*}
for any $1 \le \sigma_{1} < \sigma_{2} \le 2$, where $p_{0} = p_{0}(n,s,\tilde{s},p,\Lambda,L) \in (0,1)$ is the constant determined in Lemma~\ref{lem.wh}. Now, Lemma~\ref{tech.lemma} implies
\begin{equation*}
\sup_{B_{r}}u \le c\left(\mean{B_{2r}}u^{p_{0}}\,dx\right)^{1/p_{0}} + c\left(\frac{r^{p}}{\Phi(r)}\frac{\Phi(R)}{R^{p}}\right)^{1/(p-1)}\tail(u_{-};R),
\end{equation*}
and combining this last estimate with \eqref{wh.est} finally leads to \eqref{harnack.est}.
\end{proof}

\subsection{Nonlocal weak Harnack inequality}
We start with a different version of Caccioppoli type estimate for weak supersolutions to \eqref{mainPDE}. Its proof, which is very similar to that of the logarithmic estimate given in Lemma~\ref{lem.log}, can be performed exactly as in \cite[Lemma~5.1]{DKP14}. Therefore, we shall omit the proof of the next lemma.

\begin{lemma}
Let $u \in \mathbb{W}^{\phi,p}(\Omega)$ be a weak supersolution to \eqref{mainPDE} under assumptions \eqref{kernel}, \eqref{Dini}, \eqref{adec} and \eqref{ainc}, which is nonnegative in $B_{R}\equiv B_{R}(x_{0})\subset\Omega$. Then, with $w \coloneqq (u+d)^{(p-q)/p}$ for any $q \in (1,p)$ and $d>0$, we have 
\begin{equation}\label{ccpw}
\begin{aligned}
& \int_{B_{r}}\int_{B_{r}}|w(x)\eta(x)-w(y)\eta(y)|^{p}\frac{\phi(|x-y|)}{|x-y|^{n+p}}\,dx\,dy \\
& \le c\int_{B_{r}}\int_{B_{r}}(\max\{w(x),w(y)\})^{p}|\eta(x)-\eta(y)|^{p}\frac{\phi(|x-y|)}{|x-y|^{n+p}}\,dx\,dy \\
& \quad\; + c\left(\sup_{x\in \supp\eta}\int_{\mathbb{R}^{n}\setminus B_{r}}\frac{\phi(|x-y|)}{|x-y|^{n+p}}\,dy + d^{1-p}\frac{\Phi(R)}{R^{p}}[\tail(u_{-};R)]^{p-1}\right)\int_{B_{r}}w^{p}(x)\eta^{p}(x)\,dx
\end{aligned}
\end{equation}
for any $r \in (0,3R/4)$ and any nonnegative function $\eta \in C^{\infty}_{0}(B_{r})$, where $c = c(p,s,\tilde{s},q,L,\Lambda) > 0$.
\end{lemma}

Using the above lemma, we prove Theorem~\ref{thm.wh}.

\begin{proof}[Proof of Theorem~\ref{thm.wh}]
Let $1/2 \le \sigma' <\sigma \le 3/4$ and let $\eta \in C^{\infty}_{0}(B_{\sigma r})$ be a nonnegative cut-off function such that $\eta \equiv 1$ in $B_{\sigma'r}$ and $|D\eta| \le 4/[(\sigma-\sigma')r]$. We fix an exponent $\kappa$ such that
\begin{equation*}
1<\kappa < 
\begin{cases}
n/(n-s p) & \text{if }\ s p < n, \\
\infty & \text{if }\ s p \ge n,
\end{cases}
\end{equation*}
and then apply Theorem~\ref{thm:SoboPoin} and Corollary~\ref{cor.SP} to $v=w\eta$, where $w \coloneqq \bar{u}^{(p-q)/p} \coloneqq (u+d)^{(p-q)/p}$, which gives
\begin{equation*}
\left(\mean{B_{r}}(w\eta)^{p\kappa}\,dx\right)^{1/\kappa} \le c\frac{r^{p}}{\Phi(r)}\mean{B_{r}}\int_{B_{r}}|w(x)\eta(x)-w(y)\eta(y)|^{p}K(x,y)\,dx\,dy + c\mean{B_{r}}(w\eta)^{p}\,dx.
\end{equation*}
Moreover, we have
\begin{equation*}
\int_{B_{r}}\int_{B_{r}}(\max\{w(x),w(y)\})^{p}|\eta(x)-\eta(y)|^{p}K(x,y)\,dx\,dy \le \frac{c}{(\sigma-\sigma')^{p}}\frac{\Phi(r)}{r^{p}}\int_{B_{\sigma r}}w^{p}\,dx.
\end{equation*}
Connecting the above two estimates to \eqref{ccpw}, and using the fact that
\[ \sup_{x\in\supp\eta}\int_{\mathbb{R}^{n}\setminus B_{r}}K(x,y)\,dy \le cr^{n}\frac{\Phi(r)}{r^{p}}, \]
we get
\begin{equation*}
\left(\mean{B_{r}}(w\eta)^{p\kappa}\,dx\right)^{1/\kappa}
 \le c\left\{\frac{1}{(\sigma-\sigma')^{p}} + d^{1-p}\frac{r^{p}}{\Phi(r)}\frac{\Phi(R)}{R^{p}}[\tail(u_{-};R)]^{p-1}\right\}\mean{B_{\sigma r}}w^{p}\,dx.
\end{equation*}
Now, choosing 
\[ d = \frac{1}{2}\left(\frac{r^{p}}{\Phi(r)}\frac{\Phi(R)}{R^{p}}\right)^{1/(p-1)}\tail(u_{-};R) \]
in the above estimate and recalling the definition of $w$, we have
\begin{equation*}
\left(\mean{B_{\sigma'r}}\bar{u}^{(p-q)\kappa}\,dx\right)^{1/\kappa} \le \frac{c}{(\sigma-\sigma')^{p}}\mean{B_{\sigma r}}\bar{u}^{p-q}\,dx
\end{equation*}
whenever $1/2 \le \sigma' < \sigma \le 3/4$ and $q \in (1,p)$, where $c=c(n,p,s,\tilde{s},\Lambda,L)$. Then a standard, finite Moser iteration (see for instance \cite[Theorem~8.18]{GT} and \cite[Theorem~1.2]{Tru67}) leads to
\begin{equation*}
\left(\mean{B_{r/2}}\bar{u}^{t}\,dx\right)^{1/t} \le c\left(\mean{B_{3r/4}}\bar{u}^{t'}\,dx\right)^{1/t'}
\end{equation*}
for any $0<t'<t<\bar{t}$. 
We finally choose $t' = p_{0}$, the exponent determined in Lemma~\ref{lem.wh}, and note that, since $\bar{u}$ is a weak supersolution to \eqref{mainPDE}, estimate \eqref{wh.est} holds for $\bar{u}$. Connecting it to the above display, we arrive at
\begin{equation*}
\begin{aligned}
\left(\mean{B_{r/2}}u^{t}\,dx\right)^{1/t} 
& \le \left(\mean{B_{r/2}}\bar{u}^{t}\,dx\right)^{1/t} \le c\left(\mean{B_{3r/4}}\bar{u}^{p_0}\,dx\right)^{1/p_0}  \\
& \le c\inf_{B_{r}}\bar{u} + c\left(\frac{r^{p}}{\Phi(r)}\frac{\Phi(R)}{R^{p}}\right)^{1/(p-1)}\tail(\bar{u}_{-};R) .
\end{aligned}
\end{equation*}
Recalling the definition of $d$, we obtain \eqref{weak.harnack.est} for $t \in (p_{0},\bar{t})$. For lower values of $t$, H\"older's inequality gives the same conclusion.
\end{proof}

\subsection*{Author contributions} All authors wrote the main manuscript text and reviewed the manuscript.  

\subsection*{Conflict of interest} The authors declare that they have no conflict of interest.

\subsection*{Data availability} Data sharing not applicable to this article as no datasets were generated or analyzed during the current study.

\end{document}